\documentclass{amsart}
\usepackage{amsmath,amsthm,latexsym,amssymb,hyperref,amsfonts}
\usepackage{stmaryrd}
\usepackage{eucal}
\usepackage{mathrsfs}
\usepackage[all]{xy}
\usepackage[usenames,dvipsnames]{xcolor}
\allowdisplaybreaks
\numberwithin{equation}{section}

\newtheorem{thm}{Theorem}[section]
\newtheorem{lem}[thm]{Lemma}
\newtheorem{prop}[thm]{Proposition}
\newtheorem{cor}[thm]{Corollary}

\newtheorem{problem}[thm]{Problem}

\theoremstyle{definition}
\newtheorem{defn}[thm]{Definition}
\theoremstyle{remark}
\newtheorem{rmk}[thm]{Remark}
\newtheorem{ex}[thm]{Example}
\newtheorem{exs}[thm]{Examples}

\newtheorem{notn}[thm]{Notation}

\DeclareMathAlphabet{\mathbbe}{U}{bbold}{m}{n}

\newdir{ >}{{}*!/-10pt/@{>}}

\newcommand \Om{\Omega}

\newcommand \vp{\varphi}
\newcommand \ve{\varepsilon}

\newcommand \im{\operatorname{Im}}
\newcommand\coker{\operatorname{coker}}

\newcommand \sn[1]{(-1)\sp{#1}}
\newcommand{\si}{s^{-1}}

\newcommand\op{\mathrm{op}}

\newcommand{\ob}{\operatorname{Ob}}
\newcommand{\mor}{\operatorname{Mor}}

\newcommand\T{\mathbb T}
\newcommand\K{\mathbb K}
\newcommand{\Id}{\operatorname{Id}}

\newcommand{\C}{\mathsf C}
\newcommand{\D}{\mathsf D}
\newcommand\R{\mathsf R}
\newcommand\V{\mathsf V}
\newcommand\M{\mathsf M}
\newcommand\X{\mathcal X}
\newcommand\I{\mathcal I}
\newcommand\J{\mathcal J}

\newcommand{\W}{W}
\newcommand\Z{\mathcal Z}
\newcommand\cL{\mathcal L}
\renewcommand\R{\mathsf R}
\newcommand\cR{\mathcal R}
\newcommand{\WE}{\mathcal{W}}
\newcommand{\Fib}{\mathcal{F}}
\newcommand{\Cof}{\mathcal{C}}

\newcommand\post {\mathsf {Post}}
\newcommand\cell {\mathsf {Cell}}

\newcommand\lan{\operatorname{Lan}}
\newcommand\ran{\operatorname{Ran}}

\newcommand\map{\operatorname{Map}}
\renewcommand\hom{\operatorname{Hom}}
\newcommand\adjunct[4]{\xymatrix{#1\ar @<1.25ex>[rr]^{#3}&\perp&#2\ar @<1.25ex>[ll]^{#4}}}

\newcommand{\lra}{\to}
\newcommand{\colim}{\operatorname{colim}}
\newcommand{\Iso}{\mathrm{Iso}}

\newcommand{\Set}{\mathsf{Set}}
\newcommand{\Ch}{\mathsf{Ch}}

\newcommand{\rlp}[1]{{#1}^\boxslash}
\newcommand{\llp}[1]{{}^\boxslash\!{#1}}

\DeclareMathOperator{\sm}{\wedge}
\DeclareMathOperator{\II}{\mathbb I}
\DeclareMathOperator{\ti}{\times}
\DeclareMathOperator{\ot}{\otimes}
\DeclareMathOperator{\equal}{equal}

\renewcommand\Bar{\operatorname{Bar}}

\newcommand{\kk}{\mathbbe{k}}

\begin{document}
\title [Left-induced model structures]{Left-induced model structures and diagram categories}
\author[Bayeh]{Marzieh Bayeh}
\author [Hess]{Kathryn Hess}
\author[Karpova]{Varvara Karpova}
\author[K\c{e}dziorek]{Magdalena K\c{e}dziorek}
\author[Riehl]{Emily Riehl}
\author[Shipley]{Brooke Shipley}

\address{Dept. of Mathematics \& Statistics\\College West Building\\University of Regina\\3737 Wascana Parkway\\Regina, Saskatchewan\\S4S 0A2  Canada}
\email{bayeh20m@uregina.ca}

\address{MATHGEOM \\
    \'Ecole Polytechnique F\'ed\'erale de Lausanne \\
    CH-1015 Lausanne \\
    Switzerland}
    \email{kathryn.hess@epfl.ch}
    
\address{MATHGEOM \\
    \'Ecole Polytechnique F\'ed\'erale de Lausanne \\
    CH-1015 Lausanne \\
    Switzerland}
    \email{varvara.karpova@epfl.ch}
    
 \address{School of Mathematics and Statistics\\University of Sheffield\\ Sheffield S10 2TN\\United Kingdom}
\email{pmp10mk@sheffield.ac.uk}

\address{Dept.~of Mathematics\\Harvard University\\1 Oxford Street\\ Cambridge MA 02138\\USA}
\email{eriehl@math.harvard.edu}

\address{Department of Mathematics, Statistics, and Computer Science, University of Illinois at
Chicago, 508 SEO m/c 249,
851 S. Morgan Street,
Chicago, IL, 60607-7045, USA}
    \email{bshipley@math.uic.edu}

\date {\today }

 \keywords {Model category, weak factorization system, fibrant generation, Postnikov presentation, injective model structure.} 
 \subjclass [2010] {Primary: { 18G55, 55U35; Secondary: 18G35}}

 \begin{abstract} 
 {We prove existence results {\`a la Jeff Smith} for  
 left-induced model category structures, of which the injective model structure on a diagram category is an important example. 
 We further develop the notions of fibrant generation and Postnikov presentation from \cite{hess:hhg}, which are dual to a weak form of cofibrant generation and cellular presentation.  As examples, for $k$ a field and $H$ a differential graded Hopf algebra over $k$, we produce a left-induced model structure on augmented $H$-comodule algebras and show that the category of bounded {below} chain complexes of finite-dimensional $k$-vector spaces has a Postnikov presentation.}

To conclude, we investigate the fibrant generation of (generalized) Reedy categories.  In passing, we also consider cofibrant generation, cellular presentation, and the small object argument for Reedy diagrams.
 \end{abstract}

 \maketitle
\tableofcontents

\section{Introduction}

Let $(\M, \Fib, \Cof, \WE)$ be a model category and $\C$ a bicomplete category.  Given a pair of adjoint functors
$$\adjunct{\M}{\C}{L}{R},$$ 
there are well known conditions, such as \cite[Theorems 11.3.1 and 11.3.2]{hirschhorn}, under which there is a model structure on $\C$, which we call the \emph{right-induced model category structure}, with $R^{-1}(\WE)$, $R^{-1}(\Fib)$ as weak equivalences and fibrations, respectively. 

In this paper we study the dual situation, where one has a pair of adjoint functors
$$\adjunct{\C}{\M}{L}{R}$$
and wants to know when there is a model structure on $\C$ with $L^{-1}(\WE)$, $L^{-1}(\Cof)$ as weak equivalences and cofibrations, respectively.  We call this a \emph{left-induced} model structure.  Note that if the left-induced model structure exists, then the adjunction above is a Quillen pair with respect to the left-induced structure on $\C$ and the given model structure on $\M$.

{Right-induced model structures are common when the original model structure is \emph{cofibrantly generated} and when $\C$ satisfies appropriate smallness conditions. In this case, the small object argument provides model-theoretic factorizations for $\C$, and the existence of the right-induced model structure reduces to a fundamental ``acyclicity'' condition, which is sometimes difficult to check in practice; see e.g., \cite[Theorem 16.2.5]{may-ponto}.}

{We begin by exploring the dual  existence results  for left-induced model structures (Corollaries \ref{cor:postnikov} and \ref{cor:leftind-fibgen}), which are} expressed in terms of  either \emph{fibrant generation} or a \emph{Postnikov presentation}  of the model category $\M$.  We say that $\M$ is fibrantly generated by a pair $(\X, \Z)$ of classes of morphisms, called the \emph{generating fibrations} and \emph{generating acyclic fibrations}, if the acyclic cofibrations (respectively, cofibrations) of $\M$ are exactly the morphisms that have the left lifting property with respect to $\X$ (respectively, $\Z$).  If, in addition, each fibration (respectively, acyclic fibration) in $\M$ is a retract of the limit of a tower of morphisms built by pullback of elements of $\X$ (respectively, $\Z$), then $(\X, \Z)$ is a Postnikov presentation.

Working with fibrant generation and Postnikov presentations is a more delicate affair than the more familiar dual case. If a cocomplete category $\M$ and a set of maps $\I$ satisfy certain set-theoretical ``smallness'' conditions, then by a procedure called the \emph{small object argument} it is possible to construct a functorial factorization whose right factor has the right lifting property with respect to $\I$ and whose left factor is a  \emph{relative $\I$-cell complex}: a colimit of sequences of pushouts of maps in $\I$; {see Definition \ref{defn:cell} } or  \cite [Proposition 11.2.1]{hirschhorn}. As a consequence of this construction, for any model category $\M$ that is cofibrantly generated by a pair $(\I,\J)$ of sets of maps satisfying the smallness conditions, the cofibrations (respectively, acyclic cofibrations) are retracts of relative $\I$-cell complexes (respectively, relative $\J$-cell complexes): the small object argument implies that the pair $(\I,\J)$ defines a \emph{cellular presentation} (see Proposition \ref{prop:cofibgen-cell}).

No such general result holds in the dual case of  fibrant generation or Postnikov presentations: the \emph{cosmall object argument}, defined by dualizing the colimit constructions of the small object argument, requires  ``cosmallness'' conditions, which are {rarely} satisfied in practice; {see Remark~\ref{cosmall} though}.  For this reason, the terminology we introduce in Section \ref{sec:fib-and-post}    separates the lifting properties, the cellular presentation, and the smallness conditions that are normally unified by the adjective ``cofibrantly generated'' (see Remarks \ref{rmk:cof-gen-warning} and \ref{rmk:cof-gen-comparison}). Our motivation for persevering despite the technical difficulties presented by the theory of Postnikov presentations and fibrant generation is that left-induced model structures include interesting examples that were previously unknown, e.g., \cite[Theorem 2.10]{hess:hhg} and \cite[Theorem 6.2]{hess-shipley}).

{If the categories $\C$ and $\M$ are locally presentable, there is a more general existence theorem for left-induced model structures, Theorem \ref{thm:mr}, that does not require fibrant generation or a Postnikov presentation. In this result, \`a la Jeff Smith} {(see for example, ~\cite[Theorem 1.7]{beke}),} {the set-theoretical assumption of local presentability provides the model-theoretic factorizations. As for Smith's result, the construction of these is fairly inexplicit; here we appeal to recent work of Michael Makkai and Ji\v{r}\'i Rosick\'y \cite{makkai-rosicky}, which describes the details of the construction of these factorizations. What remains to check, again as is the case for Smith, is an ``acyclicity condition'': maps characterized by a certain lifting property must be mapped by $L$ to weak equivalences.  As in the dual setting, this can be daunting to verify in practice. However, numerous applications of Theorem \ref{thm:mr} have already been found:  Theorem \ref{thm:comodalg} in this paper, \cite{hess-shipley2}, and \cite{ching-riehl}.}

We are particularly interested in studying the injective model structure on diagram categories $\M^{\D}$, where $\M$ is a model category and $\D$ is a small category. The injective model structure is left-induced from the ``pointwise'' model structure on the category of $\ob\D$-indexed diagrams in $\M$. Applying Corollary \ref{cor:postnikov}, we {describe a criterion for the existence of the} injective model structure on  diagram categories  $\M^{\D}$, when the cofibrations of $\M$ are exactly the monomorphisms, and $\M $ admits a {(possibly trivial)} Postnikov presentation (Theorem \ref{thm:ab}). In contrast with the standard existence result, our theorem does not require the model structure on $\M$ to be combinatorial. {Because we have dropped this hypothesis, the challenge is to construct the remaining model-theoretic factorization.}

As a complement to our analysis of injective model structures, we show that if $\R$ is  a Reedy category or, more generally, a dualizable generalized Reedy category,  and $\M$ is model category {that is fibrantly generated by $(\X, \Z)$, then $\M^{\R}$, endowed with its (generalized) Reedy model structure, is also fibrantly generated by classes built naturally from $\X$ and $\Z$} (Theorems \ref{fibgen} and \ref{thm:generalizedReedy}). Moreover, if $\M$ has a Postnikov presentation, then so does $\M^\R$ (Theorem \ref{thm:reedy-post} and Remark \ref{rmk:gen-reedy-post}).

We also establish certain elementary properties of left-induced  model structures in general (Lemmas \ref{lem:left-fibgen}, \ref{left properness of C}, and \ref{lem:C V model}) and of injective model category structures in particular (Lemmas \ref{lem:fibgen-injective}, \ref{properness of MD}, and \ref{MD V model}). {However, certain} obvious and relevant questions about Postnikov presentations and fibrant generation remain open. 
For example, if $\M$ is a model category with Postnikov presentation $(\X, \Z)$, and $\C$ admits the left-induced model structure with respect to the adjunction above, then $\C$ is fibrantly generated by $\big(R\X, R\Z\big)$ by Lemma \ref{lem:left-fibgen}.  On the other hand, we do not know of reasonable conditions that imply that $\big(R\X, R\Z\big)$ is actually a Postnikov presentation of $\C$, as the set-theoretical difficulties mentioned above appear to present obstacles to obtaining a result in this direction. We hope to answer this question in the near future. 

{Other extensions of this work are possible. For example, the injective model structure on a ``generalized diagram category'' constructed from a diagram of model categories (as studied by \cite{Huettemann-Roendigs}, \cite{barwick}, and \cite{greenlees-shipley}) will be fibrantly generated if each of the model categories in the diagram is.
}

\subsection{Acknowledgements}  The authors express their deep gratitude to the Banff International Research Station for hosting the ``Women in Topology'' workshop at which much of the research presented in this article was carried out.  They also would like to thank the Clay Foundation very warmly for financing their travel to and from Banff.  The third author was supported also by the Swiss National Science Foundation, grant number 200020-144399, the fifth author by a National Science Foundation Postdoctoral Research Fellowship award number DMS-1103790, and the sixth author by the National Science Foundation, DMS-1104396. {Finally, the authors appreciate greatly the referee's helpful and constructive suggestions.}


\section{Left-induced model structures}\label{sec:left-induced}

In this section, we prove a number of {results providing conditions under which left-induced model structures exist}.  Our results are formulated in terms of either fibrant generation or Postnikov presentations of model categories, notions that we recall below.  We conclude by establishing elementary results concerning properness and enrichment of left-induced model structures. 

\begin{notn} Throughout this section, if $\X$ is a class of morphisms in a category, then $\widehat \X$ denotes its closure under retracts.
\end{notn}

\subsection{Fibrant generation and Postnikov presentations of model categories}\label{sec:fib-and-post}

We begin this section by recalling the elegant formulation of the definition of model categories in terms of weak factorization systems  due to Joyal and Tierney.  We then express the definitions of fibrant generation and of Postnikov presentations in this framework, dualizing the more familiar notions.

\begin{notn} Let $f$ and $g$ be morphisms in a category $\C$. If for every commutative diagram in $\C$
$$\xymatrix{ \cdot \ar[d]_f \ar[r]^{{a}} & \cdot \ar[d]^{g} \\ \cdot \ar[r]_{{b}} \ar@{-->}[ur]_{c} & \cdot}$$
the dotted lift $c$ exists, i.e., $gc=b$ and $cf=a$, then we write $f\boxslash g$.

If $\X$ is a class of morphisms in a category $\C$, then
$$\llp{\X}=\{ f \in \mor C\mid f\boxslash x\quad \forall x\in \X\},$$
and
$$\rlp{\X}=\{ f \in \mor C\mid x\boxslash f\quad \forall x\in \X\}.$$
\end{notn}

\begin{defn} A \emph{weak factorization system} on a category $\C$ consists of a pair $(\cL,\cR)$ of classes of morphisms in $\C$ such that
\begin{itemize}
\item any morphism in $\C$ can be factored as a morphism in $\cL$ followed by a morphism in $\cR$, and
\item $\cL = \llp{\cR}$ and $\cR= \rlp{\cL}$.
\end{itemize}
A weak factorization system $(\cL, \cR)$ is \emph{cofibrantly generated} {by} a class of morphisms $\I$ {if} $\cR= \rlp \I$ and thus $\cL =\llp {(\rlp \I)}$.  It is \emph{fibrantly generated} {by} a class of morphisms $\X$ {if} $\cL=\llp \X$ and thus $\cR=\rlp {(\llp \X)}$. 
\end{defn}

The definition of a model category can be compactly formulated in terms of weak factorization systems \cite[7.8]{joyal-tierney}, \cite[14.2.1]{may-ponto}.

\begin{defn}\label{defn:modelcat}  A \emph{model category} consists of a bicomplete category $\M$, together with three classes of morphism $\Fib$, $\Cof$, and $\WE$, called \emph{fibrations}, \emph{cofibrations}, and \emph{weak equivalences},  such that 
\begin{itemize}
\item $\WE$ {contains the identities} and satisfies ``2-out-of-3,'' i.e., given two composable morphisms $f$ and $g$ in $\WE$, if two of $f$, $g$, and $gf$ are in $\WE$, so is the third; and
\item $(\Cof, \Fib \cap \WE)$ and $(\Cof\cap \WE, \Fib)$ are both weak factorization systems.
\end{itemize}

A model category $(\M, \Fib, \Cof, \WE)$ is \emph{cofibrantly generated}  by a pair of classes of morphisms $(\I, \J)$  if the weak factorization systems $(\Cof, \Fib \cap \WE)$ and $(\Cof\cap \WE, \Fib)$ are cofibrantly generated by $\I$ and $\J$, respectively.  The elements of $\I $ and $ \J$ are then \emph{generating cofibrations}  and \emph{generating acyclic cofibrations}. 

A model category $(\M, \Fib, \Cof, \WE)$ is \emph{fibrantly generated} by a pair of classes of morphisms $(\X, \Z)$  if the weak factorization systems $(\Cof\cap \WE, \Fib)$ and $(\Cof, \Fib \cap \WE)$ are fibrantly generated by $\X$ and $\Z$, respectively.  The  elements of $\X$ and $ \Z$ are  \emph{generating fibrations}  and \emph{generating acyclic fibrations}.
\end{defn}

\begin{rmk}\label{rmk:cof-gen-warning}
Our notion of cofibrant generation is \textbf{weaker} than the usual definition (see, for example, \cite[Definition 11.1.2]{hirschhorn}) in two ways. First, we allow the generators to be a class of maps; observe that any model category is (trivially) cofibrantly generated by its cofibrations and acyclic cofibrations. The reason for this convention is that past work has shown that even trivial Postnikov presentations can have non-trivial applications; see  \cite[Theorem 6.2] {hess-shipley}.

Second, we wish to introduce separate terminology for the usual smallness conditions involving the category $\M$ and the classes $\I$ and $\J$ that permit the use of the small object argument and its consequent characterizations of the cofibrations and acyclic cofibrations. This refined terminology is necessary because  the dual conditions, appropriate to the case of fibrant generation, are not satisfied in practice. 
\end{rmk}

Frequently (see Proposition \ref{prop:cofibgen-cell}), the cofibrations and acyclic cofibrations in a cofibrantly generated model category can be characterized in terms of the following construction.

\begin{defn}\label{defn:cell} Let $\I$ be a class of morphisms in a cocomplete category $\C$. Let $Y\colon\lambda \to \C$ be a functor, where $\lambda$ is an ordinal.  If for all $\beta <\lambda$, there is a pushout
$$\xymatrix{X_{\beta +1}\ar [d]_{i_{\beta+1}\in \I}\ar [r]^{k_{\beta}\in \C} \ar@{}[dr]|(.8){\ulcorner} &Y_{\beta}\ar [d]^{}\\ X'_{\beta+1}\ar [r]^{}&Y_{\beta+1}}$$
and 
$Y_{\gamma}:=\colim _{\beta<\gamma}Y_{\beta}$ for all limit ordinals $\gamma<\lambda$, then the composition of the sequence 
$$Y_{0}\to \colim_{\lambda}Y_{\beta},$$
is a \emph {relative $\I$-cell complex}.    The class of all relative $\I$-cell complexes is denoted $\cell_{\I}$.

A \emph {cellular presentation} of a weak factorization system $(\cL, \cR)$ in a cocomplete category $\C$ consists of a class $\I$ of morphisms such that  $\cL=\widehat{\cell_{\I}}$. This condition implies that $\cR=\rlp{\I}$, so {a weak factorization systems that is cellularly presented by $\J$ is also cofibrantly generated by $\J$}.

A \emph {cellular presentation} of a model category $(\M, \Fib, \Cof, \WE)$ consists of  a pair of  classes of morphisms $(\I,\J)$ that are cellular presentations of the weak factorization systems $(\Cof, \Fib \cap \WE)$ and $(\Cof\cap \WE, \Fib)$, respectively.
\end{defn}

Any cellular presentation $(\I,\J)$ defines generating cofibrations and acyclic cofibrations for the model category $\M$. Under certain hypotheses, the small object argument (see \cite[Proposition 10.5.16]{hirschhorn} or \cite[Proposition 15.1.11]{may-ponto}) implies the converse. We say that $\M$, $\I$, and $\J$ \emph{permit the small object argument} if $\I$ and $\J$ are sets and a ``smallness'' condition is satisfied. If $\M$ is locally presentable, then the smallness condition is automatically satisfied by any generating sets $\I$ and $\J$. If $\M$ is not locally presentable, there are more general notions of smallness that may be satisfied by particular sets $\I$ and $\J$. We refer to \cite[Definition 10.4.1]{hirschhorn} or \cite[Definition 15.1.7]{may-ponto} for the definition of smallness.

\begin{prop}[{\cite[Proposition 11.2.1]{hirschhorn}}]\label{prop:cofibgen-cell} If  $(\M, \Fib, \Cof, \WE)$ is a model category that is cofibrantly generated by a pair of sets of morphisms $(\I, \J)$ and permits the small object argument, then $(\I, \J)$ is a cellular presentation of $(\M, \Fib, \Cof, \WE)$.
\end{prop}

\begin{rmk}\label{rmk:cof-gen-comparison}
A model category is cofibrantly generated in the sense of  \cite[Definition 11.1.2]{hirschhorn} if and only if it satisfies the hypotheses of Proposition \ref{prop:cofibgen-cell}.
\end{rmk}

We recall from \cite{hess:hhg} the following definition, which dualizes Definition \ref{defn:cell}.

\begin{defn}  Let $\X$ be a class of morphisms in a complete category $\C$. Let $Y\colon\lambda ^\op\to \C$ be a functor, where $\lambda$ is an ordinal.  If for all $\beta <\lambda$, there is a pullback
$$\xymatrix{Y^{\beta +1}\ar [d]_{}\ar [r]^{}\ar@{}[dr]|(.2){\lrcorner}&X'^{\beta+1}\ar [d]^{x^{\beta+1}\in \X}\\ Y^{\beta}\ar [r]_{k^{\beta}\in \C}&X^{\beta+1}}$$
and 
$Y^{\gamma}:=\lim _{\beta<\gamma}Y^{\beta}$ for all limit ordinals $\gamma<\lambda$, then the composition of the tower
$$\lim_{\lambda^\op}Y^{\beta}\to Y^{0},$$
is an \emph {$\X$-Postnikov tower}.    The class of all $\X$-Postnikov towers is denoted {$\mathsf {Post}_{\X}$}. 
 
A \emph {Postnikov presentation} of a weak factorization system $(\cL, \cR)$ in a complete category $\C$ consists of a class $\X$ of morphisms such that  $\cR=\widehat{\mathsf {Post}_{\X}}$. This implies that $\cL = \llp{\X}$; a weak factorization system with a Postnikov presentation {$\X$ is also fibrantly generated by $\X$.}

A \emph {Postnikov presentation} of a model category $(\M, \Fib, \Cof, \WE)$ consists of  a pair of  classes of morphisms $(\X,\Z)$ that are Postnikov presentations of the weak factorization systems $(\Cof\cap \WE, \Fib)$ and $(\Cof, \Fib \cap \WE)$, respectively. 
\end{defn}

{\begin{rmk}\label{rmk:postx-closed} For any class $\X$, the duals of the usual arguments show that $\post_{\X}$ is the smallest class of morphisms that is closed under composition, pullback, and inverse limits indexed by ordinals.
\end{rmk}}

\begin{rmk} Every model category $(\M, \Fib, \Cof, \WE)$ admits a trivial Postnikov presentation $(\Fib, \Fib\cap \WE)$.  On the other hand, knowledge of a particularly small and tractable Postnikov presentation, such as that discussed in Section \ref{sec:examples}, may simplify certain computations and verifications of properties of the model category, in a manner analogous to the well known, dual case of particularly nice cellular presentations.
\end{rmk}

\begin{rmk}  If $(\X, \Z)$ is a Postnikov presentation for $(\M, \Fib, \Cof, \WE)$, then $\Cof = \llp \Z$ and $\Cof \cap \WE =\llp \X$, whence $(\M, \Fib, \Cof, \WE)$ is fibrantly generated by $(\X, \Z)$.  However, if a pair of sets  $(\X, \Z)$ fibrantly generates $(\M, \Fib, \Cof, \WE)$, there are no practical conditions that guarantee that the triple $\M$, $\X$, and $\Z$ ``permit the cosmall object argument.'' In the absence of an explicit ``Postnikov-like'' construction of factorizations, we cannot conclude that $\Fib=\widehat{\mathsf {Post}_{\X}}$ and $ \Fib\cap \WE=\widehat{\mathsf {Post}_{\Z}}$. 
\end{rmk}

{
\begin{rmk}\label{cosmall}
By Gelfand duality, the category of compact Hausdorff spaces is equivalent to the opposite of the category of commutative unital $C^*$-algebras. The latter is locally presentable, which implies that the the category of compact Hausdorff spaces is ``colocally presentable.'' In particular, every object is cosmall. This means that for any set of maps $\X$, the cosmall object argument can be used to construct a functorial factorization whose left factors lift against $\X$ and whose right factors are in $\post_\X$. It follows that any set of maps gives rise to a fibrantly generated weak factorization system with a Postnikov presentation.
\end{rmk}
}

\begin{rmk}One way in which Postnikov presentations of model categories, even trivial ones, can be useful is in enabling us to construct ``explicit'' fibrant replacements,  as (retracts of) Postnikov towers, with particularly simple layers, as in  \cite[Section 7]{hess-shipley}. For example, given a Postnikov presentation for a model structure on a diagram category, one could construct explicit models for homotopy limits. 
\end{rmk}

\subsection{Existence of left-induced model structures}

We now apply fibrant generation and Postnikov presentations to prove the existence of left-induced model structures, which are defined as follows.

\begin{defn} Let
$\adjunct{\C}{\M}{L}{R}$
be an adjoint pair of functors, where $(\M, \Fib, \Cof, \WE)$ is a model category, and $\C$ is a bicomplete category.  If the triple of classes of morphisms in $\C$ 
$$\Big(\rlp {\big(L^{-1}(\Cof \cap\WE)\big)} , L^{-1}(\Cof), L^{-1}(\WE)\Big)$$
satisfies the axioms of a model category, then it is a \emph{left-induced model structure} on $\C$.
\end{defn}

\begin{rmk} If $\C$ admits a model structure left-induced from that of $\M$ via an adjunction as in the definition above, then $L \dashv R$ is a Quillen pair with respect to the left-induced model structure on $\C$ and the given model structure on $\M$.
\end{rmk}

\begin{rmk} A typical instance in which one is interested in the existence of a left-induced model structure is that of the forgetful/cofree adjunction associated to a comonad $\K$ on model category $\M$,
$$\adjunct{\M_{\K}}{\M} {U_{\K}}{F_{\K}},$$
where $\M_{\K}$ is category of $\K$-coalgebras, $F_{\K}$ is the cofree $\K$-coalgebra functor, and $U_{\K}$ is the forgetful functor.
In \cite{hess-shipley}, the authors provide conditions under which the left-induced model structure exists for such an adjunction.
\end{rmk}

\begin{lem}\label{lem:left-fibgen} Let $(\M, \Fib, \Cof, \WE)$ be a model category fibrantly generated by $(\X, \Z)$.  Any model structure on a bicomplete category $\C$ that is left-induced by an adjoint pair
$$\adjunct{\C}{\M}{L}{R}$$
is fibrantly generated by $\big(R(\X), R(\Z)\big)$.  
\end{lem}

\begin{proof} The proof proceeds by a simple game of adjunction:
$$f\in L^{-1}(\Cof)\Longleftrightarrow L(f) \in \Cof={\llp\Z}\Longleftrightarrow f\in \llp{R({\Z})},$$
and similarly for {$\X$.}
\end{proof}

General existence results for left-induced model structures can be formulated in terms of the relation between the adjunction and either fibrant generation or Postnikov presentations, as a consequence of the following theorem. The statement below is a slight variant of  that in \cite{hess:hhg}, in that retracts are allowed in the desired factorizations; the proof in \cite{hess:hhg} works in this more general setting as well.

\begin{thm}[{\cite [Theorem 5.11]{hess:hhg}}]\label{thm:left-ind}  Let $\C$ be a bicomplete category,  and let $\WE'$, $\Cof'$, $\X'$, and $\Z'$ be classes of morphisms in $\C$ satisfying the following conditions.
\begin{enumerate}
\item[(a)] The class $\WE'$ satisfies ``2-out-of-3'' and contains all identities. 
\smallskip
\item[(b)] The classes $\WE'$ and $\Cof'$ are closed under retracts.
\smallskip
\item [(c)]  $\Z'\subseteq \rlp{\Cof'}$ and $\X'\subseteq \rlp{(\Cof'\cap \WE')}$.
\smallskip
\item[(d)] $\Z'\subseteq \WE'$.
\item [(e)]  For all $f\in \mor \C$, there exist 
\begin{enumerate}
\item [(i)] $i\in \Cof '$ and $p\in \widehat{\Z'}$ such
that $f=pi$; 
\item [(ii)]$j\in \Cof' \cap \WE'$ and $q\in \widehat{\X'}$ such that
$f=qj$. 
\end{enumerate} 
\end{enumerate}
Then $(\C, \widehat {\X'}, \Cof', \WE')$ is a model category that is fibrantly generated by $(\X',\Z')$.
\end{thm}

\begin{proof}[Sketch of proof]
Conditions (b), (c), and (e) imply that $(\Cof' \cap \WE', \widehat{\X'})$ and $(\Cof', \widehat{\Z'})$ are weak factorization systems. Conditions (a) and (d) imply that $\widehat{\Z'} = \widehat{\X'}\cap \WE'$. Hence, the conditions of Definition \ref{defn:modelcat} are satisfied.
\end{proof}

In \cite{hess:hhg}, Theorem \ref{thm:left-ind} was applied to prove the existence of left-induced model structure in the presence of a Postnikov presentation.

\begin{cor}[{\cite[Corollary 5.15]{hess:hhg}}] \label{cor:postnikov}Let $(\M, \Fib, \Cof, \WE)$ be a model category with Postnikov presentation $(\X, \Z)$.
Let $\adjunct{\C}{\M}{L}{R}$ be a pair of adjoint functors, where $\C$ is a bicomplete category. If 
\begin{enumerate}
\item [(a)]$L\big(\mathsf{Post}_{R(\Z)}\big)\subseteq \WE$,
\end{enumerate}
 and for all $f\in \mor \C$
there exist 
\begin{enumerate}
\item [(b)] $i\in L^{-1}(\Cof) $ and $p\in \widehat {\mathsf{Post}_{R( \Z)}}$ such
that $f=pi$, and
\smallskip
\item [(c)]$j\in L^{{-1}}(\Cof\cap \WE)$ and $q\in \widehat{\mathsf{Post}_{R(\X)}}$ such that
$f=qj$,
\end{enumerate}
then   the model structure on $\C$ left-induced by the adjunction $L\dashv R$ exists and admits a Postnikov presentation $\big(R(\X), R(\Z)\big)$. 
\end{cor}

This corollary has been applied in \cite{hess:hhg} and \cite{hess-shipley} to prove the existence of explicit left-induced model structures. A very similar proof to that in \cite{hess:hhg} enables us to establish an existence result in the presence of {the weaker hypothesis of} fibrant generation as well.

\begin{cor}\label{cor:leftind-fibgen} Let $(\M, \Fib, \Cof, \WE)$ be a model category fibrantly generated by $(\X, \Z)$.
Let $\adjunct{\C}{\M}{L}{R}$ be a pair of adjoint functors, where $\C$ is a bicomplete category. If 
\begin{enumerate}
\item [(a)] {$(L^{-1} \Cof)\rlp{} \subseteq L^{-1} \WE$,}
\end{enumerate}
 and for all $f\in \mor \C$
there exist 
\begin{enumerate}
\item [(b)] $i\in L^{-1}(\Cof)$ and $p\in \rlp {\big(\llp{R(\Z)}\big)}$ such
that $f=pi$, and
\smallskip
\item [(c)]$j\in L^{{-1}}(\Cof\cap \WE)$ and $q\in \rlp {\big(\llp{R(\X)}\big)}$ such that
$f=qj$,
\end{enumerate}
then   the model structure on $\C$ left-induced by the adjunction $L \dashv R$ exists and is fibrantly generated by $\big(R(\X), R(\Z)\big)$. 
\end{cor} 

\begin{proof}  Apply Theorem \ref{thm:left-ind} to $\WE'= L^{-1}(\WE)$, $\Cof'=L^{-1}(\Cof)$, $\X'= \rlp {\big(\llp{R(\X)}\big)}$, and $\Z'=\rlp {\big(\llp{R(\Z)}\big)}$.  {Notice then that since $(L, R)$ is an adjoint pair, $\llp R(\Z) = L^{-1}(\llp \Z)$, so one can rewrite $\Z' \subseteq \WE'$ as in (a) above.}
\end{proof}

\begin{rmk} In both of the corollaries above, we see that condition (a) is the ``glue'' binding together the two potential weak factorization systems given by conditions (b) and (c), thus ensuring that we indeed obtain a model category.   In practice, the explicit description of Postnikov towers as limits of towers of pullbacks provides us with tools for proving condition (a) of Corollary \ref{cor:postnikov} that are not available in the more general case of Corollary \ref{cor:leftind-fibgen}.
\end{rmk}

Together with a recent result of Makkai and Rosick{\'y} \cite[Theorem 3.2]{makkai-rosicky}, Corollary \ref{cor:leftind-fibgen} implies the following existence result for left-induced model structures, when $\M$ and $\C$ are locally presentable. Recall that in a locally presentable category, any generating set of morphisms gives a cellular presentation of its weak factorization system.

{\begin{thm}\label{thm:mr} Let
$\adjunct{\C}{\M}{U}{F}$
be an adjoint pair of functors, where $\C$ is a locally presentable category, and $(\M, \Fib, \Cof, \WE)$ is a locally presentable model category that is cofibrantly generated by a pair of sets.  If  
$$\big(U^{-1}\Cof\big)\llp{} \subset U^{-1}\WE,$$ then the left-induced model structure on $\C$ exists and is cofibrantly generated by sets.
\end{thm}

\begin{proof}Under the hypotheses above, \cite[Theorem 3.2]{makkai-rosicky} implies that 
$$\Big(\llp{F(\Fib)}, \rlp {\big(\llp{F(\Fib)}\big)}\Big)\quad \text{and}\quad \Big(\llp{F(\Fib\cap \WE)}, \rlp {\big(\llp{F(\Fib\cap \WE)}\big)}\Big),$$
i.e.,
\[ \Big(U^{-1}(\Cof \cap \WE), \rlp{\big(U^{-1}(\Cof\cap\WE)\big)}\Big)\quad \text{and} \quad \Big(U^{-1}(\Cof), \rlp{\big(U^{-1}(\Cof)\big)}\Big) \] 
are weak factorization systems (cf.~\cite[Remark 3.8]{makkai-rosicky}), whence conditions (b) and (c) of Corollary \ref{cor:leftind-fibgen} hold. 
\end{proof}}

{In summary, if $\M$ is  combinatorial model category, $U \colon \C \to \M$ is a colimit preserving functor between locally presentable categories, and  $\big(U^{-1}\Cof\big)\llp{} \subset U^{-1}\WE$, then $\C$ is a combinatorial model category with the left-induced model category structure.}

\subsection{Elementary properties of left-induced structures}\label{sec:elem-prop}

In this section, we show that left-induced model structures inherit certain properness and enrichment properties.

Recall that a model structure is \emph{left proper} if the pushout of a weak equivalence along a cofibration is always a weak equivalence.

\begin{lem}\label{left properness of C}
Let $(\M, \Fib, \Cof, \WE)$ be a model category, and let $\C$ be a bicomplete category. Suppose that $\C$ admits a model structure that is left-induced from $\M$ via the adjunction
$$\adjunct{\C}{\M}{L}{R}.$$ 
If $\M$ is left proper, then $\C$ is left proper, too.
\end{lem}

\begin{proof}
Consider a pushout 
$$\xymatrix{
A \ar[r]^-{f} \ar[d]_-{g} \ar@{}[dr]|(.8){\ulcorner} & B \ar[d]^-{\overline{g}} \\
C \ar[r]^-{\overline{f}} & B \sqcup_A C
}$$
in $\C$, where $g \in L^{-1}(\WE)$ and $f \in L^{-1}(\Cof)$. Since left adjoints commute with colimits, applying $L$ gives rise to a pushout in $\M$
$$\xymatrix{
LA \ar[r]^-{Lf} \ar[d]_-{Lg} \ar@{}[dr]|(.8){\ulcorner} & LB \ar[d]^-{L\overline{g}} \\
LC \ar[r]^-{L\overline{f}} & L\big(B \sqcup_A C\big) \cong LB \sqcup_{LA} LC,
}$$
where $Lf \in \Cof$ and $Lg \in \WE$.  Since $\M$ is left proper, the morphism 
$$L\overline{g}\colon LB \to   LB \sqcup_{LA} LC\cong L\big(B \sqcup_A C\big)$$ 
is a weak equivalence in $\M$, whence $\overline{g} \in L^{-1}(\WE)$, so that $\C$ is left proper, as desired. 
\end{proof}

Extra assumptions are required for the transfer of right-properness. In Section \ref{sec:rtprop-enrich}, we prove a right properness result for diagram categories.

We refer the reader to Appendix \ref{app} for background on enriched model categories.

\begin{lem}\label{lem:C V model}
Let $(\V, \sm, \II)$ be a closed symmetric monoidal model category, and let $(\M, \Fib, \Cof, \WE)$ be a $\V$-model category. 
Suppose that $\C$ admits a model structure that is left-induced from $\M$ via the adjunction
$$\adjunct{\C}{\M}{L}{R}.$$
If $\C$ is a tensored and cotensored $\V$-category and if $L\dashv R$ is a $\V$-adjunction, then $\C$ is a $\V$-model category.
\end{lem}

\begin{proof}
We need to show that the SM7 axiom and the unit axiom from Definition \ref{Vstr}  hold in $\C$. 

Let $i\colon X \to Y$ be a cofibration in $\C$, i.e., $i\in  L^{-1}(\Cof)$, and let $j\colon V\to W$ be a cofibration in $\V$. We must show that the map 
$$j \widehat{\otimes} i \colon  V \ot Y \sqcup_{V \ot X} W \ot X \to W  \ot Y$$ 
is a cofibration in $\C$, which is a weak equivalence if either $j$ or $i$ is a weak equivalence. 
Observe that
\begin{eqnarray*} 
L(V \ot Y \sqcup_{V \ot X} W \ot X)  &\cong & L(V \ot Y) \sqcup_{L(V \ot X)} L(W \ot X)\\
& \cong& V \ot  LY \sqcup_{V \ot  LX} W \ot  LX,
\end{eqnarray*}
where the first isomorphism holds because $L$ commutes with colimits and the second isomorphism holds because $L\dashv R$ is a $\V$-adjunction (see Definition \ref{tens cotens}).

Up to isomorphism we can therefore identify the morphism $L(j \widehat{\otimes} i)$ with
$$j \widehat{\otimes} Li\colon V \ot  LY \sqcup_{V \ot  LX} W \ot  LX \to W \ot LY.$$
Since $Li\in \Cof$, and  $\M$ satisfies SM7, it follows that $j \widehat{\otimes} Li \in \Cof$, whence $j \widehat{\otimes} i$ is a cofibration in $\C$, as desired. If, in addition, $Li$ or $j$ is a weak equivalence, then by SM7 in $\M$, the map $j \widehat{\otimes} Li \in \Cof \cap \WE$, whence $j \widehat{\otimes} i $ is an acyclic cofibration in $\C$. 
 
To prove the unit axiom in $\C$, we need to check that for any cofibrant object $C$ in $\C$, the morphism $q \ot \Id_{C}\colon  Q\II \ot C \to \II \ot C$ is a weak equivalence in $\C$, where $q \colon Q\II \to \II$ denotes cofibrant replacement. By the definition of the left-induced model structure, this is equivalent to asking that the map $L(q \ot \Id_{C})\colon  L(Q\II \ot C) \to L(\II \ot C)$ be a weak equivalence in $\M$. Since $L(Q\II \ot C) \cong Q\II \ot LC$ and $L(\II \ot C) \cong \II \ot LC$, $L(q \ot \Id_{C})$ is isomorphic to  
 $$q \ot \Id_{LC}\colon  Q\II \ot LC \to \II \ot LC.$$
 Because $L$ preserves cofibrations and initial objects, $LC$ is a cofibrant object in $\M$. Since the unit axiom holds in $\M$, we conclude that $q \ot \Id_{LC}$ is a weak equivalence in $\M$; therefore, $q \ot \Id_{C}$ is a weak equivalence in the left-induced model structure on $\C$.
\end{proof}

In Section \ref{sec:rtprop-enrich} we apply Lemma \ref{lem:C V model} to establish the existence of enriched model structure on certain diagram categories.

\section{Examples}\label{sec:examples}

In this section we provide examples to illustrate the theory of section \ref{sec:left-induced}. We first give an explicit Postnikov presentation of the category of non-negatively graded chain complexes over a field, then apply Theorem \ref{thm:mr} to prove the existence of a left-induced model category structure on the category of $H$-comodule algebras over a field, where $H$ is a chain Hopf algebra of finite type.

\subsection{An explicit Postnikov presentation}

Let $\kk$ be a field, and let $\Ch$ be the category of non-negatively graded chain complexes of $\kk$-vector spaces.  This category admits a model structure whose weak equivalences are quasi-isomorphisms, whose cofibrations are degreewise monomorphisms, and whose fibrations are degreewise surjective in positive degrees \cite{dwyer-spalinski}. Let $\Ch_{\mathrm{fin}}$ denote the  full subcategory of  chain complexes that are finite dimensional in each degree.  This category is not complete or cocomplete, but does admit all limits and colimits that are finite in each degree. The model category structure of $\Ch$ therefore restricts to a model category structure on  $\Ch_{\mathrm{fin}}$ , with the same distinguished classes of morphisms, but where, {as originally defined by Quillen~\cite{quillen-homotopical}}, we assume only finite completeness  and cocompleteness. This suffices to define and study the associated homotopy category.  

Our aim in this section is to show that the model structure on $\Ch_{\mathrm{fin}}$ admits a Postnikov presentation, which provides an elementary and non-obvious characterization of the fibrations and acyclic fibrations.

Write $D^n$ for the $n$-disc and $S^n$ for the $n$-sphere, defined as follows. \[ (D^n)_\ell = \begin{cases} \kk : \ell = n, n-1 \\ 0 : \mathrm{else}\end{cases} \qquad (S^n)_\ell = \begin{cases} \kk : \ell = n \\ 0 : \mathrm{else}\end{cases}\]  It is productive to think of $S^n$ as $K(\kk,n)$ and $D^n$ as based paths in the Eilenberg Mac Lane space $K(\kk,n-1)$; a chain map $X \to D^n$ is an element of $X^*_{n-1}$ and a chain map $X \to S^n$ is an element of $Z^nX^*$, where $X^*$ denotes the dual chain complex.

Define \[ \X = \{ p_n \colon D^n \to S^n\}_{n \geq 1} \cup \{  p_0 \colon 0 \to S^0 \} \qquad \mathrm{and}\qquad \Z = \{ q_n \colon D^n \to 0 \}_{n  \geq 1}. \]

\begin{thm}\label{thm:chain} The pair $(\X,\Z)$ defines a Postnikov presentation for the model category $\Ch_{\mathrm{fin}}$
\end{thm}

The proof of Theorem \ref{thm:chain} will be given by the following two lemmas, which show that the two constituent weak factorization systems have Postnikov presentations. 

\begin{lem}\label{lem:Z-chain} The set $\Z$ defines a Postnikov presentation for the acyclic fibrations in $\Ch_{\mathrm{fin}}$.
\end{lem}
\begin{proof}
Let  $f \colon X \to Y$ be any chain map.   For each $n$, choose a basis $\beta _{n}$ of $X_{n}$.  The chain map $f$ factors as \[ X \xrightarrow{(f, \mathrm{ev})} Y \times \prod_n \prod_{\beta _{n}} D^{n+1} \xrightarrow{\mathrm{proj}} Y.\] The right factor is a pullback of products of maps $q_n$, which may be re-expressed as a limit of a tower of pullbacks of maps in $\Z$. Hence the right factor lies in $\post_\Z$, and moreover the left factor is a degreewise monomorphism. By the retract argument, if $f$ is an acyclic fibration, then $f$ is a retract of the displayed right factor. Hence, the acyclic fibrations admit a Postnikov presentation.
\end{proof}

\begin{lem}\label{lem:X-chain} The set $\X$ defines a Postnikov presentation for the fibrations in $\Ch_{\mathrm{fin}}$.
\end{lem}
\begin{proof}
Consider a chain map $f:X\to Y$.   We wish to factor $f$ as
$$X\xrightarrow j \widetilde X \xrightarrow {p } Y,$$
where $j$ is an injective quasi-isomorphism, and $p$ is an element of $\post_{\X}$.    {If $f$ is a fibration, then $f$ is a retract of $p$, whence $\Fib = \widehat{\post _{\X}}$, as is required in the definition of Postnikov presentation.}

We begin by considering two special cases.  Observe first that $S^{n} \to 0$ is an element of $\post_{\X}$ for all $n\geq 0$, since it is the pullback of
$$0\to S^{n+1}\xleftarrow{} D^{n+1}.$$
Second, it follows that $\Z\subset \post_{\X}$ and thus that $\post_{\Z}\subset \post _{\X}$ {by Remark \ref{rmk:postx-closed}}, since $D^{n} \to 0$ is equal to the composite
$$D^{n}\to S^{n}\to 0$$
 for all $n\geq 1$.  By Lemma \ref{lem:Z-chain}, we can therefore assume that $f\colon X\to Y$ is degreewise-injective.
 
 Since we are working over a field, for every $n$, there are decompositions
 $$X_{n}= H_{n}X \oplus \overline X_{n} .$$ {Note that this decomposition and similar ones below are not natural.}
 Fix a basis $\alpha_{n}$ of $\ker H_{n}f$ for every $n$, and let $\W$ denote the pullback of
 $$Y \to 0 \xleftarrow{} \prod_{n\geq 0}\prod_{\alpha _{n}}S^{n},$$
 i.e., 
 $$\W= Y\times \prod_{n\geq 0}\prod_{\alpha _{n}}S^{n}.$$
Because $S^{n}\to 0$ is an element of $\post_{\X}$ for all $n$, the induced map $f'\colon \W\to Y$ is also an element of $\post_{\X}$.
 
 For each $n$ and each $a\in \alpha_{n}$, since $a$ is a cycle that is not a boundary, there is chain map $j_{a}\colon X\to S^{n}$ such that $j_{a}(a)$ is a generator of $S^{n}$, while $j_{a}(x)=0$ if $x$ is not a multiple of $a$.  Let 
 $$j=(f, \prod_{n\geq 0}\prod_{a\in\alpha _{n}} j_{a})\colon  X \to \W$$ 
 {Because $f$ was assumed to be a degreewise injection, $j$ is also a degreewise injection. By construction, $j$ also induces a degreewise injection in homology.}
 
 We now  modify $\W$ degree by degree, to obtain a map $\tilde\jmath\colon  X \to \widetilde W$ that is also surjective in homology.  {Our  strategy is to extend $\W$, adding generators that become the images of the
cycles  in $W$ that are not hit by $j$, thus killing off these cycles in homology.} First note that there is a decomposition 
$$\W_{0}= \im H_{0}j\oplus \coker H_{0}j\oplus \overline \W_{0}.$$
{We modify $\W_0$ so as to kill $\coker H_0j$.} Let $\beta_{0}$ denote a basis of $\coker H_{0}j$. Because none of the elements of $\coker H_{0}j$ is a boundary, there is a chain map
$$g^{0}\colon \W\to \prod_{\beta_{0}}S^{0}$$
that sends any element not in $\coker H_{0}j$ to 0 and sends an element of $\beta_{0}$ to a generator of the corresponding copy of $S^{0}$ {and to zero in each of the other copies of $S^0$.}  Note that $g^{0}j=0$. 

Let $W^{0}$ denote the pullback of 
$$\W\xrightarrow{g^{0}}  \prod_{\beta_{0}}S^{0}\leftarrow 0.$$
Since $\beta_{0}$ is finite, $\prod_{\beta_{0}}S^{0}\cong \bigoplus_{\beta_{0}}S^{0}\cong \coker H_{0}j$. The induced map $p^{0}\colon W^{0}\to \W$ is an element of $\post_{X}$, and the map $j^{0}\colon X \to W^{0}$ induced by $j\colon X\to \W$ and $X\to 0$ is a degreewise injection such that $H_{0}j^{0}$ is an isomorphism, and $H_k j^0$ is an injection for all $k$.

Suppose now that for some $n\geq 0$,  the chain map $j\colon X\to \W$ can be factored as
$$X\xrightarrow {j^{n}} W^{n} \xrightarrow {p^{n}} \W,$$
where $j^{n}$ is a degreewise injection such that $H_{k}j^{n}$ is an isomorphism for all $k\leq n$ and an injection for all $k>n$.
There are decompositions
$$W^{n}_{n}=H_{n}W^{n}\oplus\overline W^{n}_{n}= \im H_{n}j^{n} \oplus  \overline W^{n}_{n}$$
and
$$W^{n}_{n+1}=H_{n+1}W^{n}\oplus\overline W^{n}_{n+1}=\im H_{n+1}j^{n} \oplus \coker H_{n+1}j^{n} \oplus \overline W^{n}_{n+1},$$ 
where every element  of $H_{n}W^{n}$ and of $H_{n+1}W^{n}$  is a cycle that is not a boundary.  In particular, the differential $d$ of $W^{n}$ satisfies $d(\overline W^{n}_{n+1})\subseteq \overline W^{n}_{n}$. We now modify $W^n_n$ by attaching a copy of $\coker H_{n+1}j^{n}$ and modify the differential $d\colon W^n_{n+1} \to W^n_n$ so that $\coker H_{n+1}j^{n}$ maps isomorphically onto this vector space. In this way, these elements disappear from $(n+1)$-st homology.

Let $\beta_{n+1}$ be a basis for $\coker H_{n+1}j^{n}$.  None of the elements of $\coker H_{n+1}j^{n}$ is a boundary, and all are cycles, so there is a chain map
$$g^{n+1}\colon X^{n}\to \prod_{\beta_{n+1}}S^{n+1}$$
that sends an element of $\beta_{n+1}$ to a generator of the corresponding copy of $S^{n+1}$ and any element not in $\coker H_{n+1}j^{n}$ to 0.   Consider the decomposition
$$X_{n+1}=H_{n+1}X\oplus \overline X_{n+1}.$$ 
Since every element of $\overline X_{n+1}$ is either a cycle that is a boundary or a noncycle, and $j^{n}$ is injective, it follows that $j^{n}(  \overline X_{n+1})\subseteq  \overline W^{n}_{n+1}$ and therefore that
$$g^{n+1}j^{n}=0\colon  X \to \prod_{\beta_{n+1}}S^{n+1}.$$

Let $W^{n+1}$ denote the pullback of    
$$W^{n}\xrightarrow{g_{n+1}}  \prod_{\beta_{n+1}}S^{n+1}\leftarrow \prod_{\beta_{n+1}}D^{n+1}.$$
Note that $W^{n+1}$ differs from $W^n$ only in level $n$, where $W^{n+1}_{n} \cong W^n_{n} \oplus \prod_{\beta_{n+1}} \kk$. The induced map $p^{n+1}\colon W^{n+1}\to W^{n}$ is an element of $\post_{\X}$, and the map $j^{n+1}\colon X \to W^{n+1}$ induced by $j\colon X\to W^{n}$ and $0\colon X\to \prod_{\beta_{n+1}}D^{n+1}$ is a degreewise injection such that $H_{k}j^{n+1}$ is an isomorphism for all $k\leq n+1$ and an injection for all $k>n+1$. 

To complete the argument, let $\widetilde W=\lim_{n}W^{n}$. Note that this limit is finite in each degree and is therefore an object in $\Ch_{\mathrm{fin}}$;  indeed, $\widetilde W_n \cong W^{n+1}_n$ since $W^k_n \cong W^{n+1}_n$ for all $k \geq n+1$. 
 Let $j\colon X \to \widetilde W$ denote the map induced by the family $j^{n}\colon  X \to W^{n}$, let $\tilde p\colon \widetilde W \to \W$ denote the composite of the $p^{n}\colon X^{n}\to X^{n-1}$, and let $p=f'\tilde p$.
 It follows that $j$ is an injective quasi-isomorphism, and $p$ is an element of $\post_{\X}$.
\end{proof}

\subsection{An application of Theorem  \ref{thm:mr}}

In this section $\Ch$ refers again to the category of non-negatively graded chain complexes over a field $\kk$.  Let $(\Ch, \Fib, \Cof, \WE)$ denote the projective model category structure on $\Ch$, as above.   It is well known that $(\Ch, \Fib, \Cof, \WE)$ is a combinatorial model category \cite[3.7]{shipley-hz}. Since we are working over a field, all objects in $\Ch$ are both fibrant and cofibrant.

The theorem below, the proof of which relies strongly on Theorem \ref{thm:mr}, plays an important role in Karpova's study of homotopic Hopf-Galois extensions of differential graded algebras in \cite{karpova}.  The mathematical concepts that are in play are the following.

\begin{defn}  Let $(\V, \otimes , \II)$ be a symmetric monoidal category.  Let $H$ be a bimonoid in $\V$.  An \emph{$H$-comodule algebra} is an $H$-comodule in the category of monoids in $\V$.
\end{defn}

\begin{rmk}  For every bimonoid $H$, there is an adjunction
$$\adjunct{\mathsf{Alg}^{H}}{\mathsf{Alg}}{U}{-\otimes H},$$
where $\mathsf{Alg}$ denotes the category of monoids in $\V$ and $\mathsf{Alg}^{H}$ the category of $H$-comodule algebras, and $U$ is the forgetful functor.  Restricting to augmented monoids, we obtain an adjunction
$$\adjunct{\mathsf{Alg}_{*}^{H}}{\mathsf{Alg}_{*}}{U}{-\otimes H}.$$

Note that  $\mathsf{Alg}^{H}$ (respectively, $\mathsf{Alg}_{*}^{H}$)  is the category of coalgebras for the comonad 
$$\K_{H}=(-\otimes H, -\otimes \delta, -\otimes \ve)$$ 
acting on $\mathsf{Alg}$ (respectively, $\mathsf{Alg}_{*}$), where $\delta$ and $\ve$ are the comultiplication and counit of $H$, respectively. 
\end{rmk}

\begin{rmk}\label{rmk:comonad-cocomplete} For use in the proof below, we recall the elementary fact that for any comonad $\K$ acting on a cocomplete category $\C$, the category $\C_{\K}$ of $\K$-coalgebras is cocomplete, with colimits created in $\C$. 
\end{rmk}

\begin{rmk} \label{rmk:ss+ar} The following observation, combining results  due to Schwede and {the last author}~\cite{schwede-shipley} and to Ad{\'a}mek and Rosick{\'y}~\cite{adamek-rosicky}, is also useful in the proof below. Let $(\M, \Fib, \Cof, \WE)$ be a combinatorial model category with a set $\J$ of generating acyclic cofibrations.  Let $\T =(T, \mu, \eta)$ be a monad acting on $\M$.  Suppose that $T$ preserves filtered colimits, and $U^{\T}(\cell_{F^{\T}\J})\subseteq \WE$.  By \cite[Lemma 2.3]{schwede-shipley},   $\M^{\T}$ admits a cellularly presented, right-induced model category structure. On the other hand, as proved in \cite[Theorem 2.78] {adamek-rosicky},  since $\M$ is locally presentable, and $\T$ preserves filtered colimits, $\M^{\T}$ is also locally presentable.  It follows that the free $\T$-algebra/forgetful-adjunction
$$\adjunct{\M}{\M^{\T}}{F^{\T}}{U^{\T}}$$ 
right-induces a combinatorial model category structure on the category $\M^{\T}$ of $\T$-algebras.
\end{rmk}

\begin{thm}\label{thm:comodalg} If $\kk$ is  a field, and $H$ is a differential graded $\kk$-Hopf algebra that is finite dimensional in each degree, then the adjunction $U\dashv (-\otimes H)$  left-induces a combinatorial model category structure on the category $\mathsf {Alg}_{*}^{H}$ of augmented right $H$-comodule algebras.
\end{thm}

 \begin{proof}  As proved in  \cite[Theorem 4.1]{schwede-shipley}, the hypotheses of  Remark \ref{rmk:ss+ar} are satisfied by the free associative algebra monad $\T_{\text{Alg}}=(T, \mu, \eta)$ acting on  $(\Ch, \Fib, \Cof, \WE)$. There is therefore a combinatorial model category structure on the category of associative chain algebras, inducing a combinatorial model structure $(\Fib_{\text{Alg}}, \Cof_{\text{Alg}}, \WE_{\text{Alg}})$ on the category of augmented, associative chain algebras $\mathsf{Alg}_{*}$, which is also trivially fibrantly generated by $(\Fib_{\text{Alg}}, \Fib_{\text{Alg}}\cap \WE_{\text{Alg}})$.
 
Since limits in $\mathsf{Alg}_{*}$ are created in $\Ch$, the proof of \cite[Lemma 6.8]{hess-shipley} implies that all limits in $\mathsf {Alg}_{*}^{H}$  exist and are created in $\Ch $ as well, since $H$ is finite dimensional in each degree. By Remark \ref{rmk:comonad-cocomplete}, we deduce that $\mathsf {Alg}_{*}^{H}$ is bicomplete.

Theorem \ref{thm:mr} now implies that, in order to conclude, it suffices to show that
$$U\Big(\rlp {\big(\llp{((\Fib_{\text{Alg}}\cap \WE_{\text{Alg}})\otimes H)}\big)} \Big)\subseteq \WE_{\text{Alg}},$$
which is equivalent, by the usual adjunction argument, to proving that  
$$\rlp{(U^{-1}\Cof_{\text{Alg}})}\subseteq U^{-1}\WE_{\text{Alg}},$$
where the elements of $\WE_{\text{Alg}}$ are quasi-isomorphisms of associative chain algebras.

We begin by showing that every element of $\rlp{(U^{-1}\Cof_{\text{Alg}})}$ is a surjection.  Let $\Bar$ denote the reduced bar construction functor, which associates a connected, coassociative chain coalgebra to any object in $\mathsf {Alg}_{*}$, and let $\Om$ denote its adjoint, the reduced cobar construction functor \cite[\S 10.3]{neisendorfer}.  By \cite[Corollary 10.5.4]{neisendorfer}, every component
$$\ve_{A}:\Omega \Bar A \to A$$
of the counit of the adjunction is a quasi-isomorphism. We now prove that if $(A, \rho)$ is an $H$-comodule algebra, then the $H$-coaction lifts to a coaction $\widetilde \rho$ on $\Omega \Bar A$, so that $\ve_{A}$ is a morphism of $H$-comodule algebras.�

The coaction $\widetilde\rho$ is equal to the composite
$$\xymatrix{\Om\Bar A \ar [rr]^(0.43){\Om\Bar \rho}\ar[rrrdd]_{\widetilde \rho}&&\Om\Bar (A\otimes H)\ar [r]^(0.45){\vp_{AW}} &\Om (\Bar A \otimes \Bar H) \ar [d]^{q}\\
&&&\Om\Bar A \otimes \Om \Bar H \ar [d]^{\Om\Bar A \otimes \ve_{H}}\\
&&&\Om \Bar A \otimes H,}$$
where $\vp_{AW}$ is the natural quasi-isomorphism of chain algebras realizing the strongly homotopy coalgebra map structure of the Alexander-Whitney map $$\Bar (A\otimes H) \to \Bar A \otimes \Bar H$$ \cite[Theorem A.11] {hess:twisting}, and $q$ is Milgram's natural quasi-isomorphism of chain algebras \cite {milgram}.

Next we explicitly define $\widetilde\rho$ on elements.
For any $a\in A$, write $\rho(a)= a_{k}\otimes h^{k}(a)$, using the Einstein summation convention, i.e., we sum over any index that is both a subscript and a superscript, as $k$ is here.  Note that since $\rho $ is a morphism of algebras, for all $a_{1},.., a_{n}\in A$,
\begin{equation}\label{eqn:h-prod}
(a_{1}\cdots a_{n})_{k}\otimes h^{k}(a_{1}\cdots a_{n})=\pm (a_{1})_{k_{1}}\cdots (a_{n})_{k_{n}}\otimes h^{k_{1}}(a_{1})\cdots h^{k_{n}}(a_{1n}),
\end{equation}
where the sign is given by the Koszul rule, i.e., a sign $\sn {pq}$ is introduced every time an element of degree $p$ commutes past an element of degree $q$.  Moreover, since $\rho$ is a chain map,
\begin{equation}\label{eqn:h-d}
(da)_{k}\otimes h^{k}(da)=da_{k}\otimes h^{k}(a) +\sn {\deg a_{k}}a_{k}\otimes dh^{k}(a)
\end{equation}
for all $a\in A$, where $d$ denotes the differentials on both $A$ and $H$.

Let $s$ denote the endofunctor on graded vector spaces specified by $(sV)_{n}=V_{n-1}$, with inverse functor $\si$. Since the chain algebra $\Omega \Bar A$ is free as an algebra on $\si (\Bar A)_{>0}$, of which a typical element is $\si(sa_{1}|\cdots |sa_{n})$, where $a_{1},...,a_{n}\in A$, it suffices to define $\widetilde \rho$ on such elements, then to extend to a morphism of algebras.  

Define $\widetilde \rho: \Omega\Bar A \to \Omega \Bar A \otimes H$ by setting
$$\widetilde \rho \big(\si(sa_{1}|\cdots | sa_{n})\big)=\pm\si\big(s(a_{1})_{k_{1}}|\cdots |s(a_{n})_{k_{n}}\big)\otimes h^{k_{1}}(a_{1})\cdots h^{k_{n}}(a_{n}),$$
where this time the sign also takes into account that $s$ and $\si$ are operators of degrees $1$ and $-1$, respectively.
Since $(\rho\otimes H)\rho=(A\otimes \delta)\rho$, where $\delta$ is the comultiplication on $H$, it follows that $(\widetilde\rho\otimes H)\widetilde\rho=(A\otimes \delta)\widetilde\rho$.  It is easy to check that $\widetilde \rho$ is a chain map as well, using the formulas in \cite[\S 10.3]{neisendorfer} for the bar and cobar differentials and identities (\ref{eqn:h-prod}) and (\ref{eqn:h-d}).
Finally, since $\ve_{A}:\Om \Bar A \to A$ is specified by 
$$\ve_{A}\big(\si(sa_{1}|\cdots | sa_{n})\big)=\begin{cases} a_{1}&: n=1\\ 0&: n>1,\end{cases}$$
it is clear that $\ve_{A}$ is a morphism of $H$-comodule algebras, with respect to $\rho$ and $\widetilde \rho$.

Let $p: (A', \rho') \to (A, \rho)$ be an element of  $\rlp{(U^{-1}\Cof_{\text{Alg}})}$.  Consider the following commuting diagram
$$\xymatrix{\kk\ar[r]\ar [d]&(A',\rho')\ar [d]^{p}\\
(\Omega\Bar A, \widetilde \rho)\ar [r]^(0.6){\ve_{A}}\ar@{-->}[ur]^{\widetilde \ve} &(A,\rho),}$$
where $\kk$ is the initial and terminal object of $\mathsf{Alg}_{*}$.  The lefthand vertical arrow is an element of $U^{-1}(\Cof_{\text{Alg}})$, since $\Omega \Bar A$ is cofibrant in $\mathsf{Alg}_{*}$; note that it is important here that we are working over a field, so that the chain complexes  underlying $A$ and $\Bar A$ are cofibrant.  The dotted lift $\widetilde\ve$ therefore exists.  Because $\ve_{A} $ is surjective,  $p$ must be as well, since $p\widetilde\ve= \ve_{A}$.
\smallskip

We next show that if $p: (A, \rho) \to \kk$ is an element of $\rlp{(U^{-1}\Cof_{\text{Alg}})}$, then $p$ is a quasi-isomorphism. Let $ \mathsf {frAlg}_{*}^{H}$ denote the full subcategory of $ \mathsf {Alg}_{*}^{H}$, the objects of which are such that the underlying algebra is free associative. Define an endofunctor $\operatorname{Cone}$ on $ \mathsf {frAlg}_{*}^{H}$  as follows.  For every $(TV,d, \rho)\in \ob \mathsf {frAlg}_{*}^{H}$,
$$\operatorname{Cone}(TV,d, \rho)=\big(T(V\oplus sV),D, \widehat \rho\,\big),$$
where $D|_{V}=d$ and $D(sv)=-s(dv) + v$ for all $v\in V$, while the algebra morphism $\widehat \rho$ is specified on generators $v\in V$ by $\widehat \rho(v)=\rho(v)$ and 
$$\widehat \rho(sv)= \sum _{i}sv_{i1}|v_{i2}|\cdots |v_{in_{i}}\otimes h_{i},$$
if $\rho (v)=   \sum _{i}v_{i1}|v_{i2}|\cdots |v_{in_{i}}\otimes h_{i}$.  It is easy to check that  $(\widehat\rho\otimes H)\widehat\rho=(A\otimes \delta)\widehat\rho$, since $(\rho\otimes H)\rho=(A\otimes \delta)\rho$, and that $\widehat \rho$ is a chain map, since $\rho$ is a chain map.  Note that $(TV,d, \rho)$ is a sub $H$-comodule algebra of $\operatorname{Cone}(TV,d, \rho)$ and that the inclusion
$$\iota: (TV, d,\rho) \hookrightarrow \operatorname{Cone}(TV,d, \rho)$$
is always an element of $U^{-1}(\Cof_{\text{Alg}})$, since we are working over a field.  Observe moreover that the unique morphism $e:\operatorname{Cone}(TV,d, \rho)\to \kk$ is always a quasi-isomorphism. 

Consider the commuting diagram 
$$\xymatrix{(\Omega\Bar A, \widetilde \rho) \ar[r]^(0.6){\ve_{A}}\ar [d]_{\iota}&(A,\rho)\ar [d]^{p}\\
\operatorname{Cone}(\Omega\Bar A, \widetilde \rho)\ar [r]^(0.7){e}\ar@{-->}[ur]^{\widehat e} &\kk.}$$
 Since $\iota \in U^{-1}(\Cof_{\text{Alg}})$, the dotted lift $\widehat e$ exists.  Because $\widehat e \iota =\ve_{A}$ and $p\widehat e= e$, both of which are quasi-isomorphisms, it follows by ``2-out-of-6''  that $\iota$, $\widehat e$ and, in particular, $p$ are all quasi-isomorphisms.

We  are finally ready to prove that every element of $\rlp{(U^{-1}\Cof_{\text{Alg}})}$ is in fact a quasi-isomorphism.  Let $p: (A', \rho') \to (A, \rho)$ be an element of $\rlp{(U^{-1}\Cof_{\text{Alg}})}$, and take the pullback in $\mathsf {Alg}_{*}^{H}$
$$\xymatrix{(A'', \rho'') \ar[r]\ar [d]_{\bar p}\ar@{}[dr]|(.2){\lrcorner}&(A',\rho')\ar [d]^{p}\\
\kk\ar [r]^{\eta} &(A,\rho),}$$ 
where $\eta$ is the unit of the algebra $A$. Observe that $\bar p\in \rlp{(U^{-1}\Cof_{\text{Alg}})}$, since any class of morphisms characterized by a right lifting property is closed under pullback.  By the previous step in the proof, $\bar p$ is a quasi-isomorphism. Moreover, $A''\cong \kk \oplus \ker p$, as chain complexes, because limits in $\mathsf {Alg}_{*}^{H}$ are created in $\Ch$. 
It follows that $\ker p $ is acyclic, and thus $p$ is a quasi-isomorphism, since there is a short exact sequence of chain complexes
\[0\to \ker p\hookrightarrow A'\xrightarrow p A\to 0.\qedhere\]  
\end{proof}

\begin{rmk} We will show in a forthcoming article that Theorem \ref{thm:comodalg} actually holds over any commutative ring. Our proof depends on a generalization of Theorem \ref{thm:mr} to the context of enriched cofibrant generation and on working with relative model structures.
\end{rmk}

\section{Injective model structures}

Let $(\M, \Fib, \Cof, \WE)$ be a model category, and let $\D$ be a small category.  Recall that the \emph{projective model structure} on the diagram category $\M^{\D}$, if it exists, is a model structure in which the fibrations and weak equivalences are determined objectwise, i.e., $\tau\colon\Phi\to \Psi$ is a fibration (respectively, weak equivalence) if and only if $\tau_{d}\colon\Phi (d) \to \Psi(d)$ is a fibration (respectively, weak equivalence) in $\M$ for every $d\in \ob \D$.  Dually,  the \emph{injective model structure} on  $\M^{\D}$, if it exists, is a model structure in which the cofibrations and weak equivalences are determined objectwise.  An (acyclic) fibration in the injective model category structure is called an \emph{injective (acyclic) fibration}.

{If $\M$ is cofibrantly generated by a pair of sets of maps that permit the small object argument,} then $\M^\D$ admits a projective model structure for any small category $\D$. The proof constructs an explicit cellular presentation for the projective model structure. In this section, we apply the notions of Postnikov presentation and fibrant generation to studying the following problem.

\begin{problem}\label{prob} Let $\D$ be a small category, and let $(\M, \Fib, \Cof, \WE)$ be a model category. When does $\M^{\D}$ admit the injective model structure?
\end{problem}

There are already well-known existence results for injective model structures in the literature.
Most significantly, {as was known to Jeff Smith,}  if $\M$ is a combinatorial model category, then $\M^{\D}$ admits the injective model structure, for all small categories $\D$; see \cite {beke}, \cite[Proposition A.2.8.2]{lurie}. Moreover, Bergner and Rezk showed in \cite[Proposition 3.15]{bergner-rezk} that if $\D$ is an elegant Reedy category, and $\M$ is the category of simplicial presheaves on some other small category, {known to admit an injective model category structure by \cite{jardine}, then the injective model structure on  $\M^{\D^{\op}}$ coincides with the Reedy model structure.}  Finally, if $\D$ is an inverse category, then $\M^{\D}$ always admits the injective model structure; see for example \cite[Theorem 5.1.3]{hovey}. Again, the injective model structure coincides with the Reedy model structure in this case.

We show here that Postnikov presentations and fibrant generation provide us with new, useful tools for proving the existence of injective model structures, which, in turn, also admit Postnikov presentations or are fibrantly generated.

\subsection {Diagrams as algebras}\label{sec:alg}

As a warm up to constructing and studying injective model structures on diagram categories, we take a new look at one existence proof for projective model structures.

Let $\C$ be any bicomplete category, and let $\D$ be a small category, with discrete subcategory $\D_{\delta}$. Let $\iota_{\D}: \D_{\delta}\hookrightarrow \D$ denote the inclusion functor.  Let
$$\lan_{\iota_{\D}}\colon\C^{\D_{\delta}} \to \C^{\D}$$ 
denote the left adjoint of the \emph {restriction functor} 
$$\iota_{\D}^{*}: \C^{\D} \to \C^{\D_{\delta}},$$ 
which is given by precomposing with $\iota_{\D}$.
The image under $\lan_{\iota_{\D}}$ of any functor $\Phi : \D_{\delta }\to \C$ is its left Kan extension along $\iota_{\D}$, which is the same as its associated \emph{copower construction}, i.e.,
$$\lan_{\iota_{\D}} (\Phi)(d)=\coprod _{d'\in \ob \D}\coprod _{f\colon d'\to d}\Phi (d').$$

The next result, certainly well known to category theorists, is quite useful for studying model structures on diagram categories.

\begin{lem}\label{lem:diag-alg} Let $\T_{\D}$ denote the monad associated to the adjunction 
$$\adjunct {\C^{\D_{\delta}}}{\C^{\D}}{\lan_{\iota_{\D}}}{\iota_{\D}^{*}} .$$  
The category $\big(\C^{\D_{\delta}})^{\T_{\D}}$ of $\T_{\D}$-algebras in $\C^{\D_{\delta}}$ is isomorphic to $\C^{\D}$.
\end{lem}

\begin{proof}
The isomorphism can be proved by inspection or by application of Beck's monadicity theorem.
\end{proof}

To illustrate the utility of this algebraic description of diagram categories and to motivate considering the coalgebraic description below, we now show that \cite[Lemma 2.3]{schwede-shipley} and Lemma \ref{lem:diag-alg} together imply easily the well-known result that every category of diagrams {taking values in a cofibrantly generated model category that permits the small object argument admits the projective model structure, which is also cofibrantly generated, has a cellular presentation, and permits the small object argument} \cite[Theorem 11.6.1]{hirschhorn}. 

\begin{defn}\label{defn:tensorclass} Let $\C$ be a cocomplete category and $\D$ a small category. The copower induces a bifunctor $-\otimes - \colon \C \times \Set^\D \to \C^\D$ defined by 
\[ X \otimes \D(d,-) := \coprod_{\D(d,-)} X,\] for any $X \in \ob\C$ and $d \in \ob \D$. For any class $\X$ of morphisms in $\C$ and any small category $\D$, let 
$$\X \otimes \mathsf D = \{ f\otimes \mathsf D(d,-)\mid f\in \X, d\in \ob \D\}.$$
\end{defn}

\begin{rmk}  Note that 
$$\X \otimes \mathsf D = \lan_{\iota_{\D}} (\X \otimes \mathsf D_{\delta}).$$
\end{rmk}

\begin{prop}\label{prop:hirschhorn} Let $\mathsf D$ be a small category. If $\mathsf M$ is a cofibrantly generated model category with  a set $\I$ of generating cofibrations  and a set $\J$ of generating acyclic cofibrations that permit the small object argument, then $\M^{\D}$ admits the projective model structure, which has a cellular presentation given by the sets $\I\otimes \D$ of generating cofibrations and $\J\otimes \D$ of generating acyclic cofibrations. 
\end{prop}

\begin{proof} The proof consists of a straightforward application of \cite[Lemma 2.3]{schwede-shipley}.  The underlying functor of the monad $\T_{\D}$ is $\iota_{\D}^{*}\circ \lan_{\iota_{\D}}$, which commutes with all colimits since both $\lan_{\iota_{\D}}$ and $\iota_{\D}^*$ are left adjoints.
The smallness condition on $\I\otimes \D$ and  $\J\otimes \D$ follows from the smallness condition on $\I$ and $\J$.

To conclude we need only to show that every relative  $(\J\otimes \D)$-cell complex (called a regular $(\J\otimes \D_{\delta})_{\T_{\D}}$-cofibration in  \cite{schwede-shipley}) is an objectwise weak equivalence, which is a consequence of the following observation.

Let $\tau\colon  \Phi\to \Psi$ be any objectwise acyclic cofibration in $\M^{\D_{\delta}}$. For every morphism $\sigma\colon \lan_{\iota_{\D}} (\Phi) \to \Omega$ in $\M^{\D_{}}$, the morphism 
$$\widehat\tau\colon  \Omega\to \Omega \coprod_{\lan(\Phi)}\lan_{\iota_{\D}} (\Psi)$$
induced by pushing $\lan_{\iota_{\D}}(\tau)$ along $\sigma$ is also an objectwise acyclic cofibration.  To check this, evaluate the pushout of functors at an object $d$ of $\D$, obtaining a pushout diagram
$$\xymatrix{
\coprod_{d'\in \ob \D}\coprod_{f\colon d'\to d} \Phi (d')\ar [d]_{\coprod_{d'\in \ob \D}\coprod_{f\colon d'\to d}\tau_{d'}}\ar [r]^-{\sigma_{d}}\ar@{}[dr]|(.8){\ulcorner}&\Omega (d)\ar[d]^{\widehat \tau_{d}}\\
\coprod_{d'\in \ob \D}\coprod_{f\colon d'\to d} \Phi (d')\ar[r]^-{\widehat\sigma_{d}}& \big(\Omega \coprod_{\lan_{\iota_{\D}}(\Phi)}\lan (\Psi)\big)(d)
}$$
in $\M$. The left-hand vertical arrow is an acyclic cofibration, since it is a coproduct of acyclic cofibrations.  It follows that $\widehat\tau_{d}$ is also an acyclic cofibration, as acyclic cofibrations are preserved under pushout. Composites of objectwise acyclic cofibrations in $\M^\D$ are clearly objectwise acyclic cofibrations, so we conclude that every relative $(\J\otimes\D)$-cell complex is an objectwise weak equivalence.

{Finally, it is clear by adjunction that $(\I \otimes\D)^\boxslash$ and $(\J\otimes\D)^\boxslash$ are the classes of objectwise acyclic fibrations and objectwise fibrations in $\M^\D$. By Proposition \ref{prop:cofibgen-cell}, the pair $(\I \otimes\D, \J\otimes\D)$ defines a cellular presentation for the projective model structure.}
\end{proof}

The adjunction of Lemma \ref{lem:diag-alg} and construction of Definition \ref{defn:tensorclass} enable us to show that injective (acyclic) fibrations are objectwise (acyclic) fibrations. This result will surprise no one familiar with injective model structures, but proofs in the literature (e.g.,  \cite[Remark A.2.8.5]{lurie})  frequently assume more restrictive hypotheses on the model category $\M$. 

\begin{lem}\label{lem:inj-fib-obj} Let $(\M,\Fib,\Cof, \WE)$ be a model category, and let $\D$ be a small category. Injective (acyclic) fibrations in $\M^\D$ are objectwise (acyclic) fibrations.
\end{lem}
\begin{proof}
Let $\tau \colon \Phi \to \Psi$ be an injective fibration, and let $\tau_d$ denote its component at $d \in \ob\D$. We must show that $i \boxslash \tau_d$ for any acyclic cofibration $i$ in $\M$. By adjunction, this is the case if and only if $i \otimes \D(d,-) \boxslash \tau$. The map $i \otimes \D(d,-)$ is an objectwise acyclic cofibration by inspection: its components are coproducts of the map $i$. Hence, this lifting problem has a solution, and we conclude.  For injective acyclic fibrations, the proof is similar, but one considers instead lifts against any cofibration $i$; alternatively, observe that injective acyclic fibrations are both injective fibrations and objectwise weak equivalences.
\end{proof}

\subsection{Diagrams as coalgebras}

The category-theoretic framework in the previous section dualizes in an obvious manner.  Let
$$\ran_{\iota_{\D}}\colon\C^{\D_{\delta}} \to \C^{\D}$$ 
denote the right adjoint of $\iota_{\D}^{*}\colon \C^{\D} \to \C^{\D_{\delta}}.$  
The image under $\ran_{\iota_{\D}}$ of any functor $\Phi \colon \D_{\delta }\to \C$ is its right Kan extension along $\iota_{\D}$, which is the same as its associated \emph{power construction},  i.e.,
$$\ran_{\iota_{\D}} (\Phi)(d)=\prod _{d'\in \ob \D}\prod _{f\colon d\to d'}\Phi (d').$$
Observe moreover that for any morphism $g\colon d\to e$ in $\D$, the induced morphism in $\C$
$$\ran_{\iota_{\D}}(\Phi)(g)\colon\ran_{\iota_{\D}} (\Phi)(d)\to \ran_{\iota_{\D}} (\Phi)(e)$$
is defined by specifying that for every $d'\in \ob \D$ and every $f\colon e\to d'$ the composite
$$\ran_{\iota_{\D}} (\Phi)(d)\xrightarrow {\ran_{\iota_{\D}}(\Phi)(g)}\ran_{\iota_{\D}}(\Phi)(e)\xrightarrow {\operatorname{proj}_{f\colon e\to d'}} \Phi (d')$$
is equal to the projection
$$\ran_{\iota_{\D}} (\Phi)(d)\xrightarrow{\operatorname{proj}_{fg\colon d\to d'}} \Phi (d').$$

\begin{rmk}\label{rmk:eta-monic}
Note moreover that for any $d\in \ob \D$ and any $\Phi\colon  \D \to \C$, the $d$-component of the $\Phi$-component of the unit map $\eta$ of the $\iota_{\D}^{*}\dashv \ran _{\iota_{\D}}$ adjunction is given by
$$(\eta_{\Phi})_{d}=\big( \Phi(f)\big)_{d', f\colon d\to d'}\colon  \Phi (d)\to \prod _{d'\in \ob \D}\prod _{f\colon d\to d'}\Phi (d'),$$
i.e., $\operatorname{proj}_{f\colon d\to d'}\circ (\eta_{\Phi})_{d}=\Phi (f)\colon \Phi (d)\to \Phi (d').$
In particular, $\eta_{\Phi}$ is an objectwise monomorphism for all $\Phi\colon  \D \to \C$. 
\end{rmk}

Lemma \ref{lem:diag-alg} dualizes to the following result.

\begin{lem}\label{lem:diag-coalg} Let $\K_{\D}$ denote the comonad associated to the adjunction 
$$\adjunct {\C^{\D}}{\C^{\D_{\delta}}}{\iota_{\D}^{*}}{\ran_{\iota_{\D}}} .$$  
The category $\big(\C^{\D_{\delta}})_{\K_{\D}}$ of $\K_{\D}$-coalgebras in $\C^{\D_{\delta}}$ is isomorphic to $\C^{\D}$.
\end{lem}

\begin{rmk} Suppose that $(\M, \Fib, \Cof, \WE)$ is a model category and that $\D$ is a small category.  If there exists a model structure on $\M^{\D}$ left-induced by the adjunction $\iota_{D}^{*}\dashv \ran_{\iota_{\D}}$ of the lemma above, then it is necessarily the injective model structure.
\end{rmk}

To prove the existence of injective model structures on diagram categories, one could attempt to exploit this comonadic description of diagram categories to prove a result dual to Proposition \ref{prop:hirschhorn}.   To do so, we begin by dualizing the construction in Definition \ref{defn:tensorclass}.

\begin{defn}\label{defn:pitchfork} 
Let $\C$ be a complete category and $\D$ a small category. The power induces a bifunctor $-\pitchfork - \colon \C \times (\Set^{\D^\op})^\op \to \C^\D$ defined by 
\[ X \pitchfork \D(-,d) := \prod_{\D(-,d)} X,\] for any $X \in \ob\C$ and $d \in \ob \D$. For any class $\X$ of morphisms in $\C$ and any small category $\D$, let 
$$\X \pitchfork \mathsf D = \{p\pitchfork \mathsf D(-,d)\mid p\in \X, d\in \ob \D\}.$$
\end{defn}

\begin{rmk}\label{rmk:pitchfork-equality}
Note that  $$\X \pitchfork \mathsf D=\ran_{\iota_{\D}}\big(\X\pitchfork \D_{\delta}\big).$$
 \end{rmk}
 
\begin{rmk}\label{rmk:XtimesD} A typical element of $\X\pitchfork \D$ corresponding to $p \colon E \to B$ in $\X$ and $d \in \ob\D$ is a natural transformation
\[ p \pitchfork \D(-,d) \colon E \pitchfork \D(-,d) \to B \pitchfork \D(-,d)\] whose component at $d' \in \ob\D$ is 
 $$\prod _{\D(d',d)} p \colon\prod_{\D(d',d)} E \to \prod_{\D(d',d)} B.$$
 \end{rmk}

 \begin{rmk}It is easy to see that  if $\X$ fibrantly generates (respectively, is a Postnikov presentation for) a weak factorization system $(\cL, \cR)$, then $\X\pitchfork \D_{\delta}$ fibrantly generates (respectively, is a Postnikov presentation for) the induced objectwise weak factorization system on $\C^{\D_{\delta}}\cong \prod_{d\in \ob \D}\C$.  
 \end{rmk}
 
 The remark above implies that the next result is  a special case of  Lemma \ref{lem:left-fibgen}.

\begin{lem}\label{lem:fibgen-injective} Let $(\M, \Fib, \Cof, \WE)$ be a model category fibrantly generated by $(\X, \Z)$, and let $\D$ be a small category. If $\M^{\D}$ admits the injective model structure, then it is fibrantly generated by $(\X\pitchfork \D, \Z\pitchfork \D)$.
\end{lem}

We do not have a general proof that if a model category $(\M, \Fib, \Cof, \WE)$ has a Postnikov presentation $(\X, \Z)$, and the injective model category structure on $\M^{\D}$ exists, then $(\X\pitchfork \D, \Z\pitchfork \D)$ is a Postnikov presentation of the injective model category structure.  {However the injective model structures covered by the existence results of Section \ref{sec:existence-injective} do have induced Postnikov presentations.}

\begin{lem}\label{lem:post-objectwise} Let $\D$ be a small category.  If $(\M,\Fib, \Cof, \WE)$ is a model category with Postnikov presentation $(\X, \Z)$, then every element of  $\post_{\X \pitchfork \D}$ (respectively, $\post_{\Z \pitchfork \D}$) is an objectwise fibration (respectively, objectwise acyclic fibration).
\end{lem}

\begin{proof} 
By the explicit description given in  Remark \ref{rmk:XtimesD}, it is obvious that elements of $X \pitchfork \D$ are objectwise fibrations. Since limits in $\M^\D$ are computed objectwise, the component at $d \in \ob \D$ of a $(\X \pitchfork \D)$-Postnikov tower is a tower in $\M$ built from the $d$-components of maps in $X \pitchfork \D$, which we know are fibrations in $\M$.  Because fibrations are characterized by a right lifting property, retracts of towers of pullbacks of products of such maps are again fibrations in $\M$.  The proof for $\Z$ is identical to that for $\X$, simply replacing fibrations with acyclic fibrations.
\end{proof}

\subsection{Existence of injective model structures}\label{sec:existence-injective}

In this section, we exploit Postnikov presentations, in order to provide conditions under which injective model structures exist.

\begin{rmk} The main theorem of \cite{hess-shipley}, which provides conditions for left transfer of model structure to categories of coalgebras over comonads, can unfortunately not be applied directly to solve Problem \ref{prob}, as a crucial axiom from \cite{hess-shipley} (Axiom (K6)) does not hold in this context.
\end{rmk}

Our partial solutions to Problem \ref{prob} rely on Corollary \ref{cor:postnikov}.   We first show that it is enough to prove that the two desired types of factorization actually exist. 

\begin{prop}\label{prop:enough-bc} Let $\D$ be a small category, and let $(\M,\Fib, \Cof, \WE)$ be a model category with Postnikov presentation $(\X, \Z)$.
If it is possible to factor any map in $\M^\D$ as 
\begin{itemize} 
\item [(a)] an objectwise cofibration followed by a map in $\widehat{\post_{\Z\pitchfork \D}}$, and as 
\item[(b)] an objectwise acyclic cofibration followed by a map in $\widehat{\post_{\X\pitchfork \D}}$, 
\end{itemize} 
then $\M^\D$ admits the injective model structure, with Postnikov presentation $(\X\pitchfork \D, \Z\pitchfork \D)$.
\end{prop}

\begin{proof} We need only to check that condition (a) of Corollary \ref{cor:postnikov} holds for the adjunction
$$\adjunct {\M^{\D}}{\M^{\D_{\delta}}}{\iota_{\D}^{*}}{\ran_{\iota_{\D}}}$$
with respect to the induced Postnikov presentation $(\X\pitchfork \D_{\delta}, \Z\pitchfork \D_{\delta}) $ of the objectwise model structure on $\M^{\D_{\delta}}$.  

By Lemma \ref{lem:post-objectwise}, every element of $\post_{\Z\pitchfork \D}$ is an acyclic fibration. Since acyclic fibrations  are, in particular, weak equivalences, we conclude that $(\Z\pitchfork \D)$-Postnikov towers are objectwise weak equivalences, as desired.
\end{proof}

\begin{rmk} The proof above illustrates the utility of Postnikov presentations.  Suppose that we knew only that $(\M,\Fib, \Cof, \WE)$ was fibrantly generated by $(\X, \Z)$,  and that any map in $\M^\D$ could be factored as
\begin{itemize} 
\item [(a)] an objectwise cofibration followed by a map in $\rlp{\big(\llp{(\Z\pitchfork \D)}\big)}$, and as 
\item[(b)] an objectwise acyclic cofibration followed by a map in $\rlp{\big(\llp{(\X\pitchfork \D)}\big)}$.
\end{itemize}
According to Corollary \ref{cor:leftind-fibgen}, to conclude that the injective model structure on $\M^\D$ exists,  we must show
$$\iota_{\D}^{*}\Big(\rlp{\big(\llp{(\Z\pitchfork \D)}\big)}\Big)\subseteq \WE,$$
 which is not as obvious as the fact that $\iota_{\D}^{*}(\post_{\Z\pitchfork\D})\subseteq \WE$ and which may not even be true in general.
\end{rmk}

We now show that if the cofibrations of $\M$ are exactly the monomorphisms, then one of the required factorization axioms always holds.

\begin{thm}\label{thm:ab} Let $\D$ be a small category, and let $(\M,\Fib, \Cof, \WE)$ be a model category with  Postnikov presentation $(\X, \Z)$ such that the cofibrations in $\M$ are exactly the monomorphisms. If  for all  morphisms $\tau\colon  \Phi \to \Psi$ in $\M^{\D}$ there exists an objectwise acyclic cofibration $\iota$ and a map $\rho\in \widehat{\mathsf{Post}_{\X\pitchfork \D}}$ such that $\tau=\rho\iota$, then $\M^{\D}$ admits an injective model structure with Postnikov presentation $(\X\pitchfork \D, \Z \pitchfork \D)$.
\end{thm}

\begin{rmk}\label{exer:mono}
The following general properties of monomorphisms, which are easy to check, are used in the proof of this theorem.  

\begin{enumerate}
\item If $\{i_{\alpha}: A_{\alpha}\to B_{\alpha}\mid \alpha \in  J\}$ is a collection of monomorphisms in a category $\C$ that admits products, then 
$$\prod_{\alpha\in  J}i_{\alpha}\colon\prod_{\alpha\in  J}A_{\alpha}\to \prod_{\alpha\in  J}B_{\alpha}$$ is also a monomorphism in $\C$.
\item If $i:A\to B$ is a monomorphism in a category $\C$ admitting pullbacks, then for every commuting diagram in $\C$
$$\xymatrix{A\ar [d]_{i}\ar [r]^{f}&C\ar [d]^{g}\\ B\ar [r]^{h}&D,
}$$
the induced morphism $(i,f): A \to B\times_{D} C$ is also a monomorphism.
\end{enumerate}
\end{rmk}

\begin{proof}[Proof of Theorem \ref{thm:ab}] By Proposition \ref{prop:enough-bc}, we need only to check that condition (b) of Corollary \ref{cor:postnikov} holds for the adjunction
$$\adjunct {\M^{\D}}{\M^{\D_{\delta}}}{\iota_{\D}^{*}}{\ran_{\iota_{\D}}}$$
with respect to the induced Postnikov presentation $(\X\pitchfork \D_{\delta}, \Z\pitchfork \D_{\delta}) $ of $\M^{\D_{\delta}}$ of the objectwise model structure on $\M^{\D_{\delta}}$.

Let $\tau\colon  \Phi \to \Psi$ be a morphism in $\M^{\D}$.  Let $*$ denote the constant diagram at the terminal object. Since $(\X\pitchfork \D_{\delta}, \Z\pitchfork \D_{\delta})$ is a Postnikov presentation for $\M^{\D_{\delta}}$, the unique morphism $\iota_{\D}^{*}\Phi\to *$ in $\M^{\D_{\delta}}$ admits a factorization
$$\xymatrix{\iota_{\D}^{*}\Phi \ar [dr]_{\theta} \ar [rr]^{!}&&{*}\\ &\Phi'\ar [ur]_{\rho}}$$
where $\theta$ is an objectwise cofibration, i.e., an objectwise monomorphism, and $\rho\in \mathsf{Post}_{\Z\pitchfork \D_{\delta}}$.  Taking transposes, we obtain a commuting diagram in $\M^{\D}$
$$\xymatrix{&\Phi\ar[dl]_{\eta_{\Phi}} \ar [dr]_-{\theta^{\sharp}} \ar [rr]^{!}&&\ran_{\iota_{\D}}{*}=*\\ \ran_{\iota_{\D}}(\iota_{\D}^{*}\Phi)\ar[rr]_{\ran_{\iota_{\D}}(\theta)}&& \ran_{\iota_{\D}}(\Phi')\ar [ur]_{\ran_{\iota_{\D}}(\rho)}}$$
where
\begin{itemize}
\item $\eta_{\Phi}$ is an objectwise monomorphism { by Remark \ref{rmk:eta-monic}},
\smallskip

\item $\ran_{\iota_{\D}}(\theta)$ is an objectwise monomorphism by Remark \ref{exer:mono}(1), using the explicit description of $\ran_{\iota_{\D}}$ above, and
\smallskip 

\item $\ran_{\iota_{\D}}(\rho)\in \ran_{\iota_{\D} }\big(\mathsf{Post}_{\Z\pitchfork \D_{\delta}}\big) \subseteq \mathsf{Post}_{\Z\pitchfork \D}.$
\end{itemize}
The inclusion in the last point follows from the definition of Postnikov towers, the fact that $\ran_{\iota_{\D}}$ is a right adjoint, {and Remark \ref{rmk:pitchfork-equality}}.  It follows that the unique morphism $\Phi\to *$ factors as an objectwise monomorphism, $\theta^{\sharp}$, followed by an element of $\mathsf{Post}_{\Z\pitchfork \D}$, namely  $\ran_{\iota_{\D}}(\rho)$.

To conclude, consider the following commuting diagram  in $\M^{\D}$.
$$\xymatrix{\Phi \ar [dr]_-{(\theta^{\sharp}, \tau)} \ar [rr]^{\tau}&&\Psi\\
& \ran_{\iota_{\D}}(\Phi')\times\Psi \ar [ur]_{\operatorname{proj}_{2}}}$$
By Remark \ref{exer:mono}(2), the natural transformation $(\theta^{\sharp}, \tau)$ is an objectwise monomorphism.  Moreover, $\operatorname{proj}_{2} \in \mathsf{Post}_{\Z\pitchfork \D}$, as it is obtained by pulling back $\ran_{\iota_{\D}}(\rho)$ along the unique morphism $\Psi \to *$, and $ \mathsf{Post}_{\Z\pitchfork \D}$ is closed under pullbacks. The factorization $\tau= \operatorname{proj}_{2}\circ(\theta^{\sharp},\tau)$  satisfies condition (b) of Corollary \ref{cor:postnikov}.
\end{proof}

\subsection{Properties of injective model structures}\label{sec:rtprop-enrich}
 
 As seen in Section \ref{sec:elem-prop}, left properness is inherited by all left-induced model structures.  We show now that right properness is inherited as well for {the injective model category structure on a diagram category $\M^{\D}$.} 
 
Recall that a model structure is \emph{right proper} if the pullback of a weak equivalence along a fibration is always a weak equivalence.

\begin{lem}\label{properness of MD}
Let $(\M, \Fib, \Cof, \WE)$ be a right proper model category, let $\D$ be a small category, and suppose that $\M^{\D}$ admits the injective model structure. 
If $\M$ is right proper, then so is $\M^{\D}$.
\end{lem}

\begin{proof}  Since limits in $\M^{\D}$ are calculated objectwise, to prove right properness of the injective model category structure on $\M^{\D}$,  it is enough to show that every injective fibration in $\M^{\D}$ is an objectwise fibration. This was done in Lemma \ref{lem:inj-fib-obj}.
\end{proof}

Let $\V$ be a closed symmetric monoidal category. Recall that if $\M$ is a $\V$-category that is tensored and cotensored over $\V$, and $\D$ is any small category, the diagram category $\M^{\D}$ inherits a natural  $\V$-enrichment and is moreover tensored and cotensored over $\V$ (cf.~Appendix \ref{app}).
 
\begin{lem}\label{MD V model}
Let $(\V, \sm, \II)$ be a closed symmetric monoidal model category. Let $\M$ be a model category, and let $\D$ be a small category. Suppose that $\M^{\D}$ admits the injective model structure.
If $\M$ is a $\V$-model category, then so is $\M^{\D}$.
\end{lem}

\begin{rmk} In the case where $\V$ is an excellent model category (\cite[Definition A.3.2.16]{lurie}) and $\M$ is combinatorial, Lurie proved in \cite[Proposition A.3.3.2]{lurie} that the injective model structure on $\M^{\D}$ exists and is a $\V$-model structure, for all small categories $\D$.
\end{rmk}

\begin{proof} By Lemma \ref{lem:C V model}, it is enough to show that the adjunction $\iota^*_{\D} \dashv \ran_{\iota_{\D}}$ is a $\V$-adjunction, which is obvious, since $\iota^*_{\D}$ is clearly a tensor $\V$-functor (cf.~Definition \ref{tens cotens} and Section \ref{app:diag-cat}). 
\end{proof}

\section{(Generalized) Reedy diagram categories and fibrant generation}
 
In this section, we show that if $\R$ is a (generalized) Reedy category, and $\M$ is a fibrantly generated model category, then the diagram category $\M^{\R}$, endowed with its Reedy model structure, is also fibrantly generated (Theorems \ref{fibgen} and \ref{thm:generalizedReedy}).  {Moreover, if $\R$ is a (classical) Reedy category and if $\M$ admits a Postnikov presentation, then so does $\M^\R$ (Theorem \ref{thm:reedy-post}). A similar result is true for generalized Reedy categories, though we neglect to give full details of the proof here (Remark \ref{rmk:gen-reedy-post}).}
 
\subsection{The case of Reedy categories}

We begin by recalling the definition of a Reedy category, as formulated in \cite{hovey}.

\begin{defn} A \emph{Reedy category} is a small category $\R$ together with a \emph{degree function} $d: \ob \R \lra \mathbb{N}$ and two wide subcategories $\R^+$ and $\R^{-}$ satisfying the following axioms.
\begin{enumerate}
\item Non-identity morphisms of $\R^+$ {strictly} raise degree.
\item  Non-identity morphisms of $\R^{-}$ {strictly} lower degree.
\item Every morphism in $\R$ factors uniquely as a morphism in $\R^{-}$ followed by a morphism in $\R^+$.
\end{enumerate}
\end{defn}

\begin{ex} An inverse or a direct category is an example of a Reedy category. In the former case, all morphisms are degree-decreasing, and in the latter case all morphisms are degree-increasing.
\end{ex}

The following constructions play an essential role in the definition of the Reedy model structure.

\begin{defn} Let $\R$ be a Reedy category, and let $\C$ be a bicomplete category. 
The \emph{$r$-th latching object} of an object $\Phi$ in $\C^{\R}$ is 
$$L_r\Phi = \colim (\R^+_{<\mathrm{deg}(r)}/r \xrightarrow U\R \xrightarrow \Phi \C ),$$
and the  \emph{$r$-th matching object}  is 
$$M_r\Phi = \lim (r/\R^{-}_{<\mathrm{deg}(r)} \xrightarrow U \R \xrightarrow\Phi \C ),$$
where $\R^+_{<\mathrm{deg}(r)}/r$ is the slice category over $r$ with objects restricted to degree less than the degree of $r$ and morphisms in $\R^+$, and $r/\R^{-}_{<\mathrm{deg}(r)}$ is defined dually, and $U$ denotes a forgetful functor from a slice category.
\end{defn}

\begin{rmk}\label{rmk:relative-matching} For every $r\in \ob \R $ and every morphism $\tau\colon  \Phi \to \Psi$, there are natural morphisms in $\C$: the \emph{relative latching map} 
$$\ell_{r}(\tau)\colon\Phi(r) \sqcup_{L_r\Phi} L_r\Psi \lra \Psi(r)$$
and the  \emph{relative matching map}
$$m_{r}(\tau)\colon\Phi (r) \lra M_r\Phi \times_{M_r\Psi} \Psi(r).$$
\end{rmk}

\begin{thm}[{\cite[Theorem 5.2.5]{hovey}}] Let $(\M, \Fib, \Cof, \WE)$ be a model category and $\R$  a Reedy category. Then $\M^{\R}$ can be equipped with a model structure (called the \emph{Reedy model structure}) where a morphism $\tau\colon \Phi \to \Psi$ is 
\begin{itemize}
\item a weak equivalence if and only if $\tau_r\in \WE$ for every $r\in \ob \R$;
\item a cofibration if and only if $\ell_{r}(\tau)\in \Cof$ for every $r\in \ob \R$; and
\item a fibration if and only if $m_{r}(\tau)\in \Fib$ for every $r\in \ob \R$.
\end{itemize}
\end{thm}

Diagram categories equipped with a Reedy model structure inherit cofibrant generation from the target category. The generating (acyclic) cofibrations are defined by a ``relative'' version of the construction given in Definition \ref{defn:tensorclass}. 

\begin{defn}\label{defn:reedy-cof-generators} Let $r$ be an object of a Reedy category $\R$. Define $\kappa_r \colon \partial \R(r,-)\hookrightarrow \R(r,-)$ to be the subfunctor of the representable functor consisting of all morphisms with domain $r$ that factor through an object of degree less than $\deg(r)$. Given a map $i \colon A \to B$ in a cocomplete category $\C$, let
\[ i \widehat{\otimes} \kappa_r \colon A \ot \R(r,-) \sqcup_{A \ot \partial\R(r,-)} B \ot \partial\R(r,-) \to B \ot \R(r,-)\]
be the map in $\C^\R$ defined by the ``pushout-product'' construction. For any set of maps $\I$ in  $\C$, write 
\[ \I \widehat{\otimes} \R = \{ i \widehat{\otimes} \kappa_r \mid i\in \I, r \in \ob\R\}.\] 

\end{defn}

\begin{thm}[{\cite[Proposition 15.6.24]{hirschhorn}}]\label{thm:Reedy-cofibgen}
If $(\M, \Fib, \Cof, \WE)$ is a model category that is  cofibrantly generated {(in the sense of Definition \ref{defn:modelcat}) }  by $(\I, \J)$, and $\R$ is  a Reedy category, then $\M^{\R}$, endowed with the Reedy model structure, is  cofibrantly generated {(in the same sense)} by $(\I \widehat{\otimes}\R, \J\widehat{\otimes}\R)$.
\end{thm}

\begin{rmk}\label{rmk:cof.gen.small}
Moreover, if the pair $(\I,\J)$ defines a cellular presentation for the model category $\M$, then $(\I \widehat{\otimes}\R, \J\widehat{\otimes}\R)$ defines a cellular presentation for the Reedy model structure (see \cite[Proposition 7.7]{riehl-verity}). We prove the dual form of this result in Theorem \ref{thm:reedy-post}.

Note that it is not automatically true under these hypotheses that $\M^\R$, $\I\widehat{\otimes}\R$, and $\J\widehat{\otimes}\R$ permit the small object argument: additional smallness conditions on the codomains of $\I$ and $\J$ are required (see \cite[Proposition 15.6.27]{hirschhorn}). However, this is irrelevant to the existence of cellular presentations for the cofibrations and acyclic cofibrations in the Reedy model structure.
\end{rmk}

We show that the dual result to this theorem holds as well: Reedy model categories also inherit fibrant generation from the target category. 

\begin{thm}\label{fibgen} If $(\M, \Fib, \Cof, \WE)$ is a model category that is fibrantly generated by sets $(\X,\Z)$, then  $\M^{\R}$, endowed with the Reedy model structure,   is fibrantly generated {by 
$\big((\X^\op \widehat{\otimes}\R^\op)^\op,(\Z^\op\widehat{\otimes}\R^\op)^\op\big)$.}
\end{thm}
\begin{proof}
The dual of a Reedy category is again Reedy, with $(\R^{\op})^+ = (\R^{-})^\op$ and $(\R^{\op})^{-} = (\R^{+})^\op$.
Similarly, $(\M, \Fib, \Cof, \WE)$  is a model category fibrantly generated by $(\X, \Z)$ if and only if  $(\M^\op,  \Cof^\op, \Fib^\op, \WE^\op)$  is a model category cofibrantly generated by $(\X^\op, \Z^\op)$
{(in the sense of Definition \ref{defn:modelcat}) }.
Moreover,  $(\M^{\R})^\op \cong (\M^{\op})^{\R^{\op}}$. 

Applying Theorem \ref{thm:Reedy-cofibgen} to $\M^{\mathrm\op}$, ${\R^{\mathrm\op}}$, and $(\X^\op, \Z^\op)$, we see that $(\M^{\op})^{\R^{\op}}$ is cofibrantly generated {(in the same sense)} by $(\X^\op \widehat{\otimes}\R^\op,\Z^\op\widehat{\otimes}\R^\op)$. Dualizing one more time, we conclude that $\M^{\R}$ is fibrantly generated by  $\big((\X^\op \widehat{\otimes}\R^\op)^\op,(\Z^\op\widehat{\otimes}\R^\op)^\op\big)$.
\end{proof}

\begin{rmk} Let us unpack the meaning of the set $(\X^\op \widehat{\otimes}\R^\op)^\op$ of generating fibrations. The ``pushout product'' construction of Definition \ref{defn:reedy-cof-generators} in $\C^\op$ dualizes to a ``pullback cotensor'' construction in $\C$ in the same way that Definition \ref{defn:pitchfork} dualizes Definition \ref{defn:tensorclass}. Substituting the Reedy category $\R^\op$ for $\R$, the maps $\kappa_r$ of definition \ref{defn:reedy-cof-generators} dualize to $\kappa^r \colon \partial\R(-,r) \hookrightarrow \R(-,r)$, where $\partial\R(-,r)$ denotes the subfunctor consisting of all morphisms with codomain $r$ that factor through an object of degree less than $\deg(r)$. Given $p \colon E \to B$ in a complete category $\C$, let 
\[ p \widehat{\pitchfork} \kappa^r \colon E \pitchfork \R(-,r) \to E \pitchfork \partial\R(-,r) \times_{B \pitchfork \partial\R(-,r)} B \pitchfork \R(-,r).\] Then for a set of maps $\X$ in $\C$, we write
\[ \X \widehat{\pitchfork} \R = \{ p \widehat{\pitchfork} \kappa^r \mid p \in \X, r \in \ob\R\} = (\X^\op \widehat{\otimes}\R^\op)^\op.\]
\end{rmk}

\begin{thm}\label{thm:reedy-post} If $(\M,\Fib,\Cof,\WE)$ is a model category that has a Postnikov presentation $(\X,\Z)$, then  $\M^\R$, endowed with the Reedy model structure, has a Postnikov presentation $(\X\widehat{\pitchfork}\R,\Z\widehat{\pitchfork}\R)$.
\end{thm}
\begin{proof}
Any morphism $\tau \colon \Phi \to \Psi$ has a canonical presentation as a relative cell complex with cells $\ell_r(\tau) \widehat{\otimes} \kappa_r$ and also a canonical presentation as  Postnikov tower with layers  $m_r(\tau) \widehat{\pitchfork} \kappa^r$ by \cite[Proposition 6.3]{riehl-verity} and its dual. If $\tau$ is a Reedy (acyclic) fibration, then each of its relative matching maps $m_r(\tau)$ is an (acyclic) fibration. By hypothesis, we know that each $m_r(\tau)$ is a retract of an $\X$- (respectively, $\Z$-) Postnikov tower. The Postnikov presentations commute with the $\widehat{\pitchfork}$ construction by the dual of \cite[Lemma 5.6]{riehl-verity}, completing the proof.
\end{proof}

\subsection{The case of generalized Reedy categories}

Generalized Reedy categories were introduced in \cite{berger-moerdijk}. A (classical) Reedy category has no non-identity automorphisms; the idea of this generalization is to allow non-trivial degree-preserving automorphisms. 

\begin{defn} A \emph{generalized Reedy category} is a small category $\R$ together with a degree function $d: \ob \R \lra \mathbb{N}$ and two wide subcategories $\R^+$ and $\R^{-}$ satisfying the following axioms.
\begin{itemize}
\item Non-{invertible morphisms of $\R^+$ strictly} raise degree.
\item  Non-{invertible morphisms of $\R^{-}$ strictly} lower degree.
\item Isomorphisms in $\R$ preserve degree.
\item $\R^+ \bigcap \R^{-} = \Iso(\R)$. 
\item Every morphism in $\R$ factors uniquely up to isomorphism as a morphism in $\R^{-}$ followed by a morphism in $\R^+$.
\item If $\theta f = f$ for $\theta \in \Iso(\R)$ and $f \in \R^{-}$, then $\theta$ is an identity.
\end{itemize}
A \emph{dualizable} generalized Reedy category is a generalized Reedy category satisfying the additional axiom:
\begin{itemize}
\item If $f \theta = f$ for $\theta \in \Iso(\R)$ and $f \in \R^{+}$, then $\theta$ is an identity.
\end{itemize}
\end{defn}

\begin{rmk} If $\R$ is a dualizable generalized Reedy category, then $\R^\op$ is as well.  Most known examples of generalized Reedy categories are dualizable.
\end{rmk}

\begin{exs} Examples of dualizable generalized Reedy categories include all groupoids, the category of finite sets, orbit categories of finite groups, and the tree category $\Omega$, presheaves on which is the category of dendroidal sets \cite{moerdijk-weiss}.
\end{exs}

\begin{defn} Let $\R$ be a generalized Reedy category.  A model category $(\M, \Fib, \Cof, \WE)$ is  \emph{$\R$-projective}  if $\M^{\mathrm{Aut}(r)}$ admits the projective model structure for every object $r \in \R$.
\end{defn}

\begin{ex} If $(\M, \Fib, \Cof, \WE)$ is cofibrantly generated and permits the small object argument, then it is $\R$-projective for any generalized Reedy category $\R$ by Proposition \ref{prop:hirschhorn} or \cite[11.6.1]{hirschhorn}.
\end{ex}

\begin{thm}[{\cite[Theorem 1.6]{berger-moerdijk}}]\label{modelstructure} 
If $(\M, \Fib, \Cof, \WE)$ is an $\R$-projective model category, and $\R$ is a generalized Reedy category, then  $\M^{\R}$ admits a model structure where a map $\tau\colon  \Phi \to \Psi$ is a
\begin{itemize}
\item weak equivalence if and only if for every $r \in \ob\R$, $\tau_r: \Phi(r) \lra \Psi(r)$ is a weak equivalence in the projective model structure on $\M^{\mathrm{Aut}(r)}$;
\item cofibration if and only if for every $r \in \ob\R$, the relative latching map $\ell_{r}(\tau)$ is a cofibration in the projective model structure on  $\M^{\mathrm{Aut}(r)}$;
\item fibration if and only if for every $r \in \ob\R$, the relative matching map $m_{r}(\tau)$ is a fibration in the projective model structure on   $\M^{\mathrm{Aut}(r)}$.
\end{itemize}
\end{thm}

\begin{thm}\label{thm:berger-moerdijk} If $\R$ is a generalized Reedy category and $(\M, \Fib, \Cof, \WE)$ is cofibrantly generated  {(in the sense of Definition \ref{defn:modelcat}) }by $(\I,\J)$, then $\M^{\R}$, equipped with a model structure of Theorem\ \ref{modelstructure}, is cofibrantly generated {(in the same sense)} by $(\I\widehat{\otimes}\R, \J\widehat{\otimes}\R)$.
\end{thm}

\begin{proof} Consider a morphism $\tau$ in $\M^\R$. By adjunction $(\I \widehat{\otimes} \kappa_r) \boxslash \tau$ if and only if $\I \boxslash m_r(\tau)$, where the relative matching map for $\tau$ is defined as in Remark \ref{rmk:relative-matching} (see \cite{riehl-reedy} for more details). This lifting problem has a solution if and only if $m_r(\tau)$ is an acyclic fibration in $\M$, which is the case if and only if it is an acyclic fibration in the projective model structure on $\M^{\mathrm{Aut}(r)}$. Thus, $(\I \widehat{\otimes}\R)^\boxslash$ is the class of Reedy acyclic fibrations. The same argument works for $\J\widehat{\otimes}\R$.
\end{proof}

We now  show that the generalized Reedy model structure is fibrantly generated, provided $(\M, \Fib, \Cof, \WE)$ is a fibrantly generated model category, and $\R$ is dualizable. The model structure from Theorem \ref{modelstructure} is not self dual, so we must construct a ``dual'' version of it, using the injective model structure on an automorphism category of any object in $\R$.

\begin{defn} Let $\R$ be a generalized Reedy category.  A model category $(\M, \Fib, \Cof, \WE)$ is  \emph{$\R$-injective}  if $\M^{\mathrm{Aut}(r)}$ admits the injective model structure for every object $r \in \R$.
\end{defn}

\begin{ex} If $(\M, \Fib, \Cof, \WE)$ is a combinatorial model category, then it is  an $\R$-injective model category for any generalized Reedy category $\R$ \cite [Proposition A.2.8.2]{lurie}. 
\end{ex}

\begin{thm}\label{modstr2}
If $(\M, \Fib, \Cof, \WE)$ is an $\R$-injective model category, and $\R$ is a dualizable generalized Reedy category, then  $\M^{\R}$ can be equipped with a model structure where a map $\tau\colon  \Phi \to \Psi$ is a
\begin{itemize}
\item weak equivalence if and only if for every $r \in \ob\R$, $\tau_r: \Phi(r) \lra \Psi(r)$ is a weak equivalence in the injective model structure on $\M^{\mathrm{Aut}(r)}$;
\item cofibration if and only if for every $r \in \ob\R$, the relative latching map $\ell_{r}(\tau)$ is a cofibration in the injective model structure on $\M^{\mathrm{Aut}(r)}$;
\item fibration if and only if for every $r \in \ob\R$, the relative matching map $m_{r}(\tau)$ is a fibration in the injective model structure on  $\M^{\mathrm{Aut}(r)}$.
\end{itemize}
\end{thm}

\begin{proof} If $(\M, \Fib, \Cof, \WE)$ is $\R$-injective, then $(\M^\op,  \Cof^\op, \Fib^\op, \WE^\op)$ is $\R^\op$-projective.  We can therefore apply Theorem \ref{modelstructure} to $(\M^\op,  \Cof^\op, \Fib^\op, \WE^\op)$ and $\R^\op$.  Dualizing the resulting model structure on $(\M^{\op})^{\R^{\op}}\cong (\M^{\R})^\op$ gives rise to the desired model structure on $\M^{\R}$, since relative latching maps and relative matching maps are mutually dual.
\end{proof}

As was the case in the proof of Theorem \ref{fibgen}, we can prove the following theorem by dualizing Theorem \ref{thm:berger-moerdijk}.

\begin{thm}\label{thm:generalizedReedy} If $\R$ is a dualizable generalized Reedy category, and $(\M, \Fib, \Cof, \WE)$ is an $\R$-injective model category that is fibrantly generated by $(\X,\Z)$, then the model structure on $\M^{\R}$ of Theorem \ref{modstr2} is fibrantly generated by $(\X\widehat{\pitchfork}\R,\Z\widehat{\pitchfork}\R)$.
\end{thm}

\begin{rmk}\label{rmk:gen-reedy-post}
Theorem \ref{thm:reedy-post} can be extended to generalized Reedy categories and the model structure of Theorem \ref{modstr2} because analogous results to  \cite[Proposition 6.3]{riehl-verity} are true in this context. The proofs however are not entirely straightforward. See \cite{riehl-reedy}.
\end{rmk}

\appendix
\section{Enriched model categories}\label{app}

We recall from e.g.,  \cite{kelly} and \cite[Chapters 3, 10]{riehl}, those elements of the theory of enriched categories and the compatibility with model structure that are essential for understanding the proof of Lemma \ref{lem:C V model}.  We then describe the particular case of diagrams in an enriched category, which we need for Lemma \ref{MD V model}.

\subsection{{Enriched categories, functors, adjunctions}}

\begin{defn}\label{Vstr}
Let $(\V, \sm, \II)$ denote a closed symmetric monoidal model category. 
\begin{itemize}
\item A category $\M$ is \emph{enriched over $\V$} if there exists a bifunctor 
$$\map_{\M}\colon  \M^{\op} \ti \M \to \V\colon  (X,Y) \mapsto \map_{\M}(X,Y)$$
together with collections of morphisms in $\V$ 
$$c_{X,Y,Z}\colon  \map_{\M}(Y,Z) \sm \map_{\M}(X,Y) \to \map_{\M}(X,Z)$$
for all $X,Y,Z \in \M$, called \emph{compositions}, and
$$i_X\colon  \II \to \map_{\M}(X,X)$$
for all $X \in \M$, called \emph{units}, such that the composition is associative and unital.
\item A $\V$-enriched category $\M$ is \emph{tensored over $\V$} if there is a bifunctor 
$$- \ot - \colon  \V \ti \M \to \M \colon  (V, X) \mapsto V \ot X$$
with natural isomorphisms $\map_{\M}(V \ot X, Y) \cong \hom_{\V}\big(V,\map_{\M}(X,Y)\big)$,
for all $X,Y \in \M$ and $V \in \V$.

\item A $\V$-enriched category $\M$ is \emph{cotensored over $\V$} if there is a bifunctor 
$$( - )^{(-)} \colon   \V^{\op} \ti \M \to \M \colon  (V, X) \mapsto X^V,$$
with natural isomorphisms
$\map_{\M}(X,Y^V) \cong \hom_{\V}\big(V,\map_{\M}(X,Y)\big)$,
for all $X,Y \in \M$ and $V \in \V$. Note that if $\M$ is tensored and cotensored over $\V$, then for all $V \in \V$, the functors $- \ot V \colon  \M \to \M$ and $( - )^{V} \colon  \M \to \M$ form an adjunction.

\item If $\V$ is a monoidal model category and $\M$ is a model category that is tensored over $\V$, then $\M$ \emph{satisfies SM7} if \[j\widehat{\otimes} i\colon V \ot Y \sqcup_{V \ot X} W \ot X \to W\otimes Y\] is cofibration whenever $i\colon X\to Y$ is a cofibration in $\M$ and $j$ is a cofibration in $\V$.  If $i$ or $j$ is also a weak equivalence, then $j\widehat{\otimes} i$ must also be a weak equivalence.

\item The \emph{unit axiom} holds in $\M$ if, for any cofibrant object $X$  in $\M$, the map $\xymatrix{Q\II \ot X \ar[r]^-{q \ot \Id} & \II \ot X}$ is a weak equivalence in $\M$, where $q\colon Q(-) \to (-)$ denotes the cofibrant replacement functor and natural transformation.
\end{itemize}
\end{defn}

If a model category $\M$ satisfies all properties of Definition \ref{Vstr} with respect to a closed symmetric monoidal model category $(\V, \sm, \II)$, then $\M$ is a \emph{$\V$-model category}.

\begin{defn}\label{Vfctor}
Let $\M$ and $\M'$ be $\V$-enriched categories. A \emph{$\V$-functor} $F\colon \M \to \M'$ consists of a function 
$$F\colon  \ob \M \to \ob \M'\colon  X \mapsto FX$$
and a collection of morphisms in $\V$
$$F_{X,Y}\colon  \map_{\M}(X,Y) \to\map_{\M'}(FX,FY)$$
for all $X,Y \in \M$, called the \textit{components}, such that
$$ c \circ \big(F_{Y,Z} \sm F_{X,Y} \big) = F_{X,Z} \circ c  \mbox{ and } F_{X,X} \circ i_{X} = i_{FX}$$
for all $X, Y, Z \in \M$.
\end{defn}

\begin{defn}\label{tens cotens}
Let $(\V, \sm, \II)$ be a closed symmetric monoidal category. An adjunction $\adjunct \C\M LR$ between two $\V$-categories is a \emph{$\V$-adjunction} if it satisfies one of the following equivalent conditions. 
\begin{enumerate}
\item The left adjoint $L$ is a \emph{tensor} $\V$-\emph{functor}, i.e., there are natural isomorphisms $L(V \ot X) \cong V \ot LX$, for all $V \in\ob \V$ and $X \in \ob\C$ {that are associative and unital with respect to the monoidal structure on $\V$.}
\item The right adjoint $R$ is a \emph{cotensor} $\V$-\emph{functor}, i.e., there are natural isomorphisms $(RY)^V \cong R(Y^V)$, for all $V \in \ob\V$ and $Y \in \ob\M$, {associative and unital with respect to the monoidal structure on $\V$.}
\item The adjunction is \emph{$\V$-enriched}, i.e., there are natural isomorphisms 
$$\map_{\M}(LX,Y)\cong \map_{\C}(X, RY),$$ 
for all $X \in \ob\C$ and $Y \in \ob\M$.
\end{enumerate}
\end{defn}

\subsection{Enriched diagram categories}\label{app:diag-cat}

Let $\V$ be a complete closed symmetric mon\-oid\-al category. If $\M$ is a {complete and cocomplete, tensored and cotensored} $\V$-enriched category, and $\D$ is a small category, then the diagram category $\M^{\D}$  is naturally $\V$-enriched, as well as tensored and cotensored over $\V$.

\begin{defn}\label{MD MapSp}
Let $\Phi, \Psi: \D \to \M$ be functors. The \emph{mapping object $\map_{\M^{\D}}(\Phi, \Psi)$} is the end of the bifunctor $\map_{\M}\big(\Phi(-), \Psi(-)\big)\colon \D^{\op} \ti \D \to \V$, i.e., 
$$\map_{\M^{\D}}(\Phi,\Psi)= \int_{d \in \D} \map_{\M}(\Phi(d), \Psi(d)).$$
This end exists, since $\D$ is small and $\V$ is complete.
\end{defn}

More precisely, the object $\map_{\M^{\D}}(\Phi,\Psi)$ is constructed as an equalizer
$$\xymatrix
@!C=6cm 
{\equal \Big( \prod_{d \in \D} \map_{\M}\big(\Phi(d), \Psi(d)\big) \ar@<3pt>[r] \ar@<-3pt>[r] & \prod_{d, d' \in \D} \prod_{\alpha \in \D(d, d')} \map_{\M}\big(\Phi(d), \Psi(d')\big) \Big)
}$$ 
in $\V$. 

Tensors and cotensors for $\M^{\D}$ over $\V$ are defined objectwise, i.e., 
$$(V \otimes \Phi)(d)=V \ot \Phi (d)\quad\text{and}\quad(\Phi^{V})(d)=\Phi(d)^{V}$$ 
for all $\Phi \in \ob\M^{\D}$, $V\in \ob \V$, and $d\in \ob \D$.

 \bibliographystyle{amsplain}
\bibliography{cmc}
\end{document}